\documentclass{amsart}
\usepackage{amscd,amsmath,amssymb,amsthm,latexsym,pinlabel}

\newcommand{\R}{{\mathbb R}^n}

\newcommand{\Om}{\Omega}

\newtheorem{thm}{Theorem}
\newtheorem{lem}[thm]{Lemma}
\newtheorem{prop}[thm]{Proposition}

\theoremstyle{definition}
\newtheorem*{defn}{Definition}
\newtheorem*{rem}{Remark}
\newtheorem*{ack}{Acknowledgements}

\begin{document}

\title[Bubbles-2013]{Soap bubbles and isoperimetric regions in the product of a closed manifold with Euclidean space}

\author[Jes\'us Gonzalo]{Jes\'us Gonzalo P\'erez}
\address{Departamento de Matem\'aticas,
Universidad Aut\'onoma de Ma\-drid, 28049 Madrid, Spain.}
\email{jesus.gonzalo@uam.es}
\thanks{Partially supported by
         grant MTM2004-04794 and MTM2011-22612 from MICINN, Spain.}

\date{}

\begin{abstract}
For any closed Riemannian manifold $X$ we prove that large isoperimetric regions in $X\times\R$ are of the form $X\times (\mbox{\rm Euclidean ball})$. We first show that isoperimetric boundaries in such ambient manifold are very regular, and then
obtain apriori estimates for CMC hypersurfaces leading to the result. We prove that if $X$ has non-negative Ricci curvature then the only soap bubbles enclosing a large volume are the products $X\times (\mbox{\rm Euclidean sphere})$. We give an example of a surface $X$, with Gaussian curvature negative somewhere, such that the product $X\times {\mathbb R}$ contains \emph{stable} soap bubbles of arbitrarily large enclosed volume which do not even project surjectively onto the $X$ factor.
\end{abstract}

\maketitle


\section{Introduction}\label{introduction}

\begin{defn}
Given a Riemannian manifold $M$, possibly with boundary, and a positive number $v$, the \textbf{isoperimetric problem} asks for a region $\Om\subset M$ whose volume is $v$ and whose perimeter is minimal among all regions of volume $v$ in~$M$.
\end{defn}

Solutions to the isoperimetric problem, if they exist, are called \textbf{isoperimetric regions}. Their boundaries are called \textbf{isoperimetric boundaries}; they are differentiable hypersurfaces except perhaps at some singular points that together make a closed subset which is either empty or of codimension at least 8 in the ambient manifold \cite{Morgan03}. More precisely, these boundaries are smooth away from the singular points and from $\partial M$, and they are of class at least ${\mathcal C}^1$ at the points where they touch $\partial M$, see e.g.~\cite{Gonzalez}.
Moreover, they have constant mean curvature in the smooth part away from~$\partial M$.

\begin{defn}
A \textbf{soap bubble} in $M$ is any smooth embedded hypersurface $S\subset M$ which has constant mean curvature and is the boundary of some smooth domain.
\end{defn}

While closely related, the two concepts are not equivalent. An isoperimetric boundary will not be a soap bubble if it has singular points or if it touches the boundary of the ambient manifold. A soap bubble may not minimize area among hypersurfaces enclosing the same volume.

The purpose of the present paper is to study the shape of soap bubbles and isoperimetric regions in Riemannian products $X\times\R$, where $X$ is a closed, connected Riemannian manifold of any dimension. Our first result gives symmetry and regularity for these objects, large or small. To make a precise statement we fix some conventions. We write $B(y,r)$ for the Euclidean ball with center $y$ and radius $r$ in ${\mathbb R}^n$, write $\overline{B}(y,r)$ for the closed ball, and $S(y,r)$ for the Euclidean sphere.

\begin{defn}
A subset $E\subset X\times {\mathbb R}^n$ is \textbf{normalized} if there is an open set $A\subseteq X$, a function $u:A\to {\mathbb R}^+$, and a point $y\in{\mathbb R}^n$, such that $E$ is the union of the coaxial balls $\{ x\}\times B\big( y,u(x)\big)$, as $x$ ranges over~$A$, or the closure of such union. The \textbf{rotated graph} of $u$ is the union of the coaxial spheres $\{ x\}\times S\big( y,u(x)\big)$ as  $x$ ranges over~$A$. The \textbf{symmetry axis} of these sets is $X\times \{y\}$.
\end{defn}

We now state the symmetry and regularity result.

\begin{thm}\label{sym-smooth}
Every isoperimetric region in $X\times\R$ is bounded and normalized. Every isoperimetric region in the compact manifold $X\times
\overline{B}(y, r)$  is, up to a translation parallel to the $\R$ factor, normalized with symmetry axis $X\times \{ y\}$. In
both cases its boundary is, after deleting the symmetry axis, the rotated graph of a ${\mathcal C}^1$
function $u:A\to{\mathbb R}^+$ which is smooth in all of $A$ for isoperimetric boundaries in $X\times\R$, and smooth
in $\{ u<r\}\subset A$ for isoperimetric boundaries in $X\times \overline{B}(y, r)$.

Any soap bubble in $X\times\R$  is, after deleting the symetry axis, the rotated graph of a smooth function.
\end{thm}

A consequence is that singular points on an isoperimetric boundariy, if any,  can only exist where said boundary
meets its symmetry axis.

Isoperimetric regions always exist if the ambient manifold is compact. For the non-compact manifold $X\times\R$ one can use the argument in \cite[page 129]{Morgan}, which provides existence on ambient spaces that have an isometry action with compact quotient. In the present paper we give a direct existence proof for $X\times\R$ by showing that, for fixed volume $v$ and large radius $r$, there are isoperimetric regions of volume $v$ in $X\times\overline{B}(y,r)$ that do not reach $X\times S(y,r)$ at all. Details are given in Section~\ref{sec-existence}.

Our two main results are the following.

\begin{thm}\label{large-isop}
Large isoperimetric regions in $X\times{\mathbb R}^n$ are of the form $X\times (\mbox{\rm ball})$.
\end{thm}

\begin{thm}\label{large-bubble}
If  $X$ has $\,\mbox{\rm Ric}\geq 0$, then soap bubbles in $X\times\R$ with large enclosed volume are of the form $X\times (\mbox{\rm sphere})$.
\end{thm}

Several authors have studied isoperimetric regions in product spaces. Wu-Yi Hsiang and Wu-Teh Hsiang \cite{HsHs} determined them in the product of two hyperbolic spaces. Duzaar and Steffen \cite{Duzaar} proved Theorem~\ref{large-isop} for the cylinder spaces $X\times{\mathbb R}$. Pedrosa and Ritor\'e \cite{PedroRit} proved Theorem~\ref{large-isop} for $S^1\times\R$ as well as the analogous result for the product of $S^1$ with a hyperbolic space.  Ritor\'e and Vernadakis \cite{RitVernad} have obtained a proof of Theorem~\ref{large-isop} in the spirit of Geometric Measure Theory. Theorem~\ref{large-isop} is also true \cite{Gonzalo}  when the ambient
anifold is the product $X\times{\mathbb H}^n$.

\begin{defn}
Let $M$ be a Riemannian manifold, and let $S\subset M$ be a soap bubble with unit normal $\nu$ and
second fundamental form $II$. Let $\,\mbox{\rm Ric}\,$ be the Ricci tensor of $M$. The
\textbf{index form} of $S$ is the following
quadratic form acting on functions $g:S\to{\mathbb R}$:
\[
Q(g)\; =\; \int_S \big(\; |\nabla^S g|^2+P\, g^2\;\big)\, d\,\mbox{\rm area}
\quad ,\quad P\; :=\; -\mbox{\rm Ric}(\nu ,\nu )-|II|^2
\; . \]
The \textbf{Jacobi equation} is $\Delta^Sg-P\, g=0$. We say that $S$ is \textbf{stable} if
its index form is positive definite on the functions with zero average over~$S$, except those
of the form $\langle\nu ,\xi\rangle$ where $\xi$ is a Killing vector field on~$M$.
\end{defn}

While isoperimetric boundaries are \textit{global} minima of area for fixed enclosed volume, stability is the
necessary condition for a \textit{local} minimum. It is natural to ask whether the condition $\,\mbox{\rm Ric}\geq 0$
in Theorem~\ref{large-bubble} can be replaced by the hypothesis of large stable soap bubble. The answer is negative.

\begin{thm}\label{teo-estables}
There exist surfaces $X^2$, with Gauss curvature negative somewhere, such that in the manifold $X\times{\mathbb R}$ we find a one-parameter family $\{ S_v\}_{0< v<\infty}$ of soap bubbles with the following properties:
\begin{enumerate}
\item The enclosed volume of $S_v$ is $v$.
\item There is a positive lower bound for the mean curvatures of the $S_v$.
\item No $S_v$ projects surjectively onto~$X$.
\item There is a $v_0>0$ such that all soap bubbles $S_v$ with $v\geq v_0$ are stable.
\end{enumerate}
\end{thm}

 The fact that the surfaces $S_v$ with $v\geq v_0$ are stable implies that it is possible to get them by ``inflating $S_{v_0}$'', so that they will never ``burst'' during the process. There are examples of this phenomenon where $X$ is a surface in ${\mathbb R}^3$ with the induced metric; then the product $X\times{\mathbb R}$ is a cylinder in ${\mathbb R}^4$. One such cylinder is
sketched in Figure~\ref{cil}; the cylinder's profile is a surface with a thin neck where Gaussian curvature is negative; three
soap bubbles $S_v$ are shown, getting larger but never exiting the domain $\, (\mbox{\rm neck})\times{\mathbb R}$.

\begin{figure}[ht]
\includegraphics[scale=0.4]{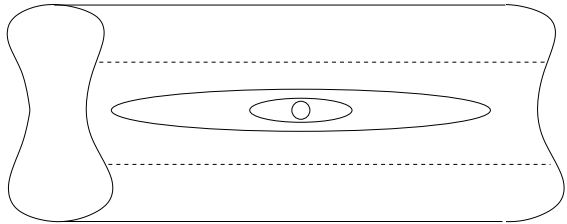}
\caption{Bubbles on a cylinder in ${\mathbb R}^4$}
\label{cil}
\end{figure}

Since $X$ is bounded and $\R$ is infinitely large, we can think of $X\times\R$ as a ``slightly
thickened Euclidean space''. Then Theorem~\ref{large-isop} exhibits large isoperimetric regions in this space as ``slightly thickened Euclidean isoperimetric regions'', and Theorem~\ref{large-bubble} gives a condition under which soap bubbles have the same
behavior. The mean curvature of spheres in $\R$ is proportional to $\, 1/\mbox{\rm radius}$, hence proportional to $(\mbox{\rm volume})^{-1/n}$. We expect a similar estimate in $X\times\R$; yet the bubbles $S_v$ of Theorem~\ref{teo-estables} become arbitrarily large while their mean curvature decreases to a positive constant. All these ideas are reflected in the following theorem, which is essential in the proof of Theorems \ref{large-isop} and~\ref{large-bubble}.

\begin{thm}\label{teo-estimates}
An isoperimetric region in $X\times\R$ of volume $v$  has the
following bound for the mean curvature of its boundary:
\begin{equation}\label{isop-H}
H\leq n\, \big(\omega_n |X|\big)^{1/n}\, v^{-1/n}
\; , \end{equation}
where $|X|$ denotes the Riemannian volume of $X$.

Given $v>0$, for $r$ sufficiently large (depending on $v$), isoperimetric regions
in $X\times\overline{B}(y,r)$ of volume $v$ satisfy
the following bound:
\begin{equation}\label{b-isop-H}
H\leq 2\, n\, \big(\omega_n |X|\big)^{1/n}\, v^{-1/n}
\; . \end{equation}
Suppose that $S\subset X\times\R$ is the closure of the rotated graph
of $u:A\to{\mathbb R}^+$, that it is not too small, and that
one of the following conditions is satisfied:
\begin{list}{}{}
\item[--] either: $S$ is isoperimetric in $X\times\R$ or in
some large $X\times\overline{B}(y,r)$,
\item[--] or: $S$ is a soap bubble and $X$ has $\,\mbox{\rm Ric}\geq 0$,
\end{list}
then we have the estimates:
\begin{equation}\label{H-osc-bound}
\max u-\min u\leq\mbox{\rm const}\quad ,\quad H\leq\frac{\mbox{\rm const}}{\max u}
\; . \end{equation}
For a large soap bubble in $X\times\R$ of volume $v$,
and without the condition $\,\mbox{\rm Ric}\geq 0$, we only have the weaker estimates:
\begin{equation}\label{H-max-bound}
\max u\leq \mbox{\rm const}\cdot v^{1/n}\quad ,\quad H\leq\mbox{\rm const}
\; . \end{equation}
The four constants depend only on $n$ and $X$.
\end{thm}

The paper is organized as follows. The boundedness part of Theorem~\ref{sym-smooth}
is proved in Section~\ref{sec-monotonicity}, the symmetry part in Section~\ref{sec-symmetry},
and the regularity part in Section~\ref{sec-regularity}. In Section~\ref{sec-estimates-1} 
we prove the part of Theorem~\ref{teo-estimates} about isoperimetric boundaries, and
in Section~\ref{sec-estimates-2} the part for soap bubbles. In 
Section~\ref{sec-existence} we prove existence of isoperimeric regions of every volume. 
The main theorems (\ref{large-isop}) and (\ref{large-bubble}) are proved 
in Section~\ref{sec-main}, based on the estimate 
(\ref{H-osc-bound}) in Theorem~\ref{teo-estimates} and a gradient bound
from the Appendix. In Section~\ref{sec-estables} we construct the families $\{ S_v\}$
of Theorem~\ref{teo-estables}. In Section~\ref{sec-stability} we show that in some of
these families all large bubbles are stable.

\begin{ack}
I owe special thanks to Wu--Yi Hsiang for
communicating this beautiful problem to me.  Frank Morgan gave me hints
as to why the examples of Theorem~\ref{teo-estables} should exist.  Bruce
Kleiner and Antonio C\'ordoba taught me some of the techniques used here. At
times I received important support from Lars Kadison. The encouragement
from Antonio Ros, Manuel Ritor\'e, Robert Kusner, and many others has also been very important.
\end{ack}


\section{Monotonicity formula}\label{sec-monotonicity}

In this section we shall obtain a lower estimate  (\ref{area-estimate}) for area and use it to prove that every isoperimetric region in $X\times\R$ is bounded. We shall also use (\ref{area-estimate}) in later sections.

We use the monotonicity formula proved in 
\cite[pages 483-484]{KKS}, which is true for all hypersurfaces
of constant mean curvature. Their proof is for a Euclidean
ambient space. We adapt it here
and  we indicate the small changes needed to
make it work in all ambient Riemannian manifolds.

For each $m$ we use ${\mathcal H}^m$ to denote $m$-dimensional Hausdorff measure.

\begin{defn}
Let $M$ be a Riemannian manifold of dimension $d$ and let $S\subset M$ be a hypersurface. By \textbf{mean curvature} of $S$ we mean the scalar function $H:S\to{\mathbb R}$ such that for any deformation $\{ S_t\}$ of any compact piece $S_0\subset S$ we have:
\begin{equation}\label{def-H}
\left.\frac{d}{dt}\right|_{t=0} {\mathcal H}^{d-1}(S_t)\; =\; \int_{S_0}\xi\cdot (H\nu )\, d{\mathcal H}^{d-1}
+\int_{\partial S_0}(\xi\cdot\eta )\, d{\mathcal H}^{d-2}
\; ,\end{equation}
where $\xi$ is the field of the velocities of motion for each point during the deformation, $\nu$ is a unit normal for $S$, and $\eta$ is the outer conormal to $\partial S_0$ in~$S$.

For example, the unit sphere in ${\mathbb R}^d$ has $H=d-1$ with respect to the outer unit normal.
\end{defn}

In \cite{KKS} they consider a smooth piece of hypersurface in $S\subset {\mathbb R}^d$ that
is of the form  $S=B(z_0,r_1)\cap \partial\Om$ for some region $\Om$,
and whose mean curvature is a positive constant $H$. The point $z_0$ is assumed to lie on $S$. Then
define, for $0<s<r_1$, the following objects:
\begin{eqnarray*}
S(s) &=& S\cap  \overline{B}(z_0,s)\; ,
\\   U(s)&=& \Om\cap \overline{B}(z_0,s)\qquad\mbox{\rm (a solid)}\; ,
\\   Q(s)&=&\Om\cap\partial B(z_0,s)\;\quad\mbox{\rm (a spherical piece)}\; ,
\\ \nu_s &=& \mbox{\rm the outer unit normal along }\;\partial U(s)\; ,
\end{eqnarray*}
and obtain a differential inequality satisfied by the area function
$a(s)\equiv{\mathcal H}^{d-1} \big( S(s)\big)$, and use it to estimate
$a(s)$ from below.

Choose orthonormal coordinates $z_1,...,z_d$ centered at $z_0$ and
consider the vector field ${\bf V}\equiv
z_1\,\partial_{z_1}+\cdots +z_d\,\partial_{z_d}$, whose flow is
$\varphi_t (z)=e^t\, z$ and whose divergence is $d$. By formula (\ref{def-H}), compute:
\begin{eqnarray*}
(d-1)\, a(s) &=&\left.\frac{d}{dt}\right|_{t=0}
e^{(d-1)t}\, a(s) \; =\;\left.\frac{d}{dt}\right|_{t=0}
{\mathcal H}^{d-1}\big( \varphi_t(S(s))\big)\; =\\
&=&\int_{S(s)}H\,\nu\cdot{\bf V} +
\int_{\partial S(s)}\eta\cdot{\bf V}\; =\;\\
&=& H\,\int_{\partial U(s)}{\bf V}\cdot\nu_s -H\,\int_{Q(s)}{\bf
V}\cdot\nu_s +\int_{\partial S(s)}\eta\cdot{\bf V}
\; . \end{eqnarray*}
Then estimate the three summands in the last expression:
\begin{eqnarray*}
H\,\int_{\partial U(s)}{\bf V}\cdot\nu_s &=& d\,
H\,{\mathcal H}^{d}(U(s))\;\leq\; d\,\omega_d\, H\, s^d\; ,
\\  -H\,\int_{Q(s)}{\bf V}\cdot\nu_s &\leq & 0\; ,
\\   \int_{\partial S(s)}\eta\cdot{\bf V} &\leq & s\,\frac{d}{ds}\; a(s)
\; .  \end{eqnarray*}
The third inequality follows from the coarea formula for the
function $\sqrt{y_1^2+\cdots +y_d^2}$.

We now have $(d-1)\, a(s)\leq  d\,\omega_d\, H\,
s^d+s\,\frac{d}{ds}\; a(s)\,$, or equivalently:
\[
\frac{d}{ds}\; \big( s^{1-d}\, a(s)\big)\geq -d\,\omega_d\, H\;
.\]
In the case $H=0$ this differential inequality says that $s^{1-d}\,
a(s)$ is a monotone increasing function. This is why it is called
monotonicity formula.

We also have  $\lim\limits_{s\to 0} s^{1-d}\,
a(s)=\omega_{d-1}\,$, which combined with the monotonicity formula
gives the following lower bound for area:
\[
a(s)\;\geq\; \big(\,\omega_{d-1}-d\,\omega_d\, H\, s\, \big)\, s^{d-1}\; .
\]
For a non--Euclidean ambient space $M$ the above proof needs the
following modifications.  Choose $(z_1,...,z_d)$ to be canonical
coordinates at $z_0$, i.e.\ coordinates for which the Christoffel
symbols vanish at $z_0$. As long as we keep $s$ small, the identities used
above are all true in an approximate way. As examples: while in the Euclidean case
we had $\nabla{\bf V} =\mbox{\rm id}$, now we have
$\nabla{\bf V} =\mbox{\rm id}+ \mbox{\rm O}(s)$;  while in the Euclidean
case ${\mathcal H}^{d}\big( B(z_0,s)\big) =\omega_d\, s^d$, now it is
${\mathcal H}^{d}\big( B^M(z_0,s)\big) =\big( 1+\mbox{\rm O}(s)\big)\,\omega_d\, s^d$.

Notice also that the calculation carries through for isoperimetric boundaries, provided $z_0$ is not a singular point.

Then for small $s$ we get $\frac{d}{ds}\; \big( s^{1-d}\, a(s)\big)\geq
-c_1\, H$ \ where \ $c_1$ \ is some positive constant close to $d\,\omega_d$ in value, and we
obtain our desired lower estimate for area:
\begin{equation}\label{area-estimate}
a(s)\;\geq\; \big(\,\omega_{d-1}-c_1\, H\, s\, \big)\, s^{d-1}\; .
\end{equation}
In a complete manifold with bounded geometry (as is
$X\times\R$) the constant $c_1$ and the radius of a
ball where the above proof is valid may be chosen the same for all
points $z_0$. Then formula (\ref{area-estimate}) provides a
range of radii for which we have a lower area bound near every regular point of $S$. The larger
$H$ is, the shorter that range of radii is:
formula (\ref{area-estimate}) is useful only in combination
with some upper bound for~$H$.

\begin{prop}
Every isoperimetric region in $X\times\R$ is bounded.
\end{prop}

Let $d$ be the dimension of $X\times\R$. An isoperimetric boundary $\partial\Om$ has constant mean curvature. No matter how large this constant is, it has a fixed value and
thus provides a positive (if small) value $r_0$ and a constant $\varepsilon >0$ such that ${\mathcal H}^{d-1}\big(\partial\Om
\cap B^{X\times\R}(z_0,r_0)\big)>\varepsilon$ for every non-singular $z_0\in\partial\Om$. If $\Om$ were unbounded then $\partial\Om$ would be unbounded and it would contain an infinity of non-singular points $z_j$ with pairwise distances all greater than $2r_0$. But then the intersections $\partial\Om \cap B^{X\times\R}(z_j,r_0)$would be pairwise disjoint, the area of $\partial\Om$ would be infinite and $\Om$ could not be isoperimetric.


\section{Symmetry}\label{sec-symmetry}

In this section we prove the symmetry part of Theorem~\ref{sym-smooth}.

If $S\subset X\times\R$ is a soap bubble, we can use A. D. Alexandrov's reflection method,
as described e.g.\ in \cite{Hopf}, to prove that $S$ is a
union of coaxial spheres $\{ x\}\times S\big( y,u(x)\big)$ as $x$ ranges over the image of $S$ under the projection $X\times\R \to X$.

This description forbids, in particular, that some
parts of $S$ be surrounded by others. This is
a rather obvious consequence of the maximum principle because
the mean curvature is \textit{the same constant\/} in all connected
components. In the case $n=1$, assuming that the axis is $X\times\{ 0\}$, the soap bubble
is the union of the graphs of $u$ and  $-u$.

Consider now the case of an isoperimetric region in $X\times\R$ or
in $X\times\overline{B} (y,r)$. There are several symmetrization procedures associated with the
names of Steiner and Schwarz, see for instance  
\cite[page 78]{BZ}. All have the effect of preserving the volume of a
(sufficiently smooth)  set without increasing its boundary area.

We consider here the following symmetrization procedure
in an arbitrary ambient manifold $M$. Fix a
Killing vector field $\bf V$ which admits an orthogonal
hypersurface $M_1\subset M$.  If $\Om$ is the region which is to
be symmetrized, then for each orbit $\gamma$ of $\bf V$
one replaces the intersection $\gamma\cap\Om$ with a segment
$\gamma_{\Om}\subset\gamma$ centered at the point $\gamma\cap
M_1$ and having the same one--dimensional measure
as~$\gamma\cap\Om$. If $\gamma\cap\Om$ is empty, then let
$\gamma_\Om$ be also empty. The symmetrized set
\[ S\Om\;\stackrel{\mbox{\rm\scriptsize
def}}{=}\;\bigcup_{\gamma}\gamma_{\Om}
\] has the same volume as $\Om$ and is symmetric with respect to
$M_1$. We claim that if $\Om$ is sufficiently regular then
the boundary area of $S\Om$ is at most that of~$\Om\,$.

For regions with enough regularity one has three equivalent
notions of boundary area:
\begin{list}{}{}
\item[--] The standard
area of the regular part of $\partial\Om\,$.
\item[--] Perimeter, see e.g.~\cite{Giusti}.
\item[--] {\bf Minkowski content,} \ defined as:
\[ \lim\limits_{h\to 0}\displaystyle\frac{1}{h} \big(\, \mbox{\rm Vol}\,
(\Om^h)-\mbox{\rm Vol}\, (\Om )\,\big)\quad ,\quad\mbox{\rm with \ }
\Om^h=\{\; z\;\; |\;\;\mbox{\rm dist}\, (z,\Om )\leq h\;\}\;
.\]
\end{list}
Isoperimetric regions have enough regularity so that these three
notions coincide, see e.g.~\cite{Maggi}.

For such regions we can now explain why $\mbox{\rm area}\, (\partial
S\Om)\leq\mbox{\rm area}\, (\partial\Om)$. Using Minkowski content
to compute boundary area, the claimed inequality follows from
$(S\Om )^h\subseteq S(\Om^h)\,$. The proof of this inclusion in
\cite[pages 78-79]{BZ} only requires that the flow of $\bf
V$ preserve distance and one--dimensional measure, hence it applies
to the general setting we have described. See also  
\cite[page 203]{Ros} for explicit pictures.

The existence of the pair ${\bf V},M_1$ provides local
coordinates $v_1,...,v_m$ in $M$  with respect to which the metric is
expressed as  $g\equiv g_0+g_{mm}\, dv_m^2\,$, where $g_0$ is a
metric on $(v_1,...,v_{m-1})$--space and the function $g_{mm}$ is independent of
$v_m$. Then the argument in  \cite[pages 108-111]{BZ} applies to
show that $S\Om$ has strictly less boundary area than $\Om$ unless
$\Om$ satisfies the following two conditions:
\begin{list}{}{}
\item[--] $\Om$ was already symmetric to start with (with respect to some image
of $M_1$ under the flow of $\bf V$). \item[--] $\Om$ is ``convex
in the direction of $\bf V$''. This means that each orbit of $\bf
V$ intersects $\Om$ in an orbit segment or the empty set.
\end{list}

\vspace{2mm}

We apply these conclusions to $M=X\times\R$ and
choose $\bf V$ to be any constant vector field along the $\R$ factor.

In the case of a region $\Om$ which is isoperimetric
in $X\times\R$, we conclude there is a point
$y\in\R$ such that $\Om$ is symmetric with
respect to all hypersurfaces
\[ X\times\mbox{\rm (Euclidean hyperplane through
}y) \] and convex in the direction of the $\R$
factor, hence a union of coaxial balls with $X\times\{ y\}$ as common axis.

In the case $\Om$ is isoperimetric in $X\times\overline{B}(y,r)$,
we let $P\subset\R$ be the Euclidean hyperplane through
$y$  orthogonal to the direction of $\bf V$  and choose $M_1=X\times P$
as hypersurface orthogonal to~$\bf V$. We notice
two properties of $X\times \overline{B}(y,r)$: it is symmetric
with respect to $M_1$ and
intersects any orbit of $\bf V$ in a line segment.
They imply that for every $\Om\subset X\times\overline{B}(y,r)$ the symmetrized region
$S\Om$ is completely contained in $X\times \overline{B}(y,r)$.
We conclude that there is a
point $y'\in B(y,r)$ such that $\Om$ is
symmetric with respect to all hypersurfaces
\[ X\times\mbox{\rm (Euclidean hyperplane through
}y')\; ,\] and convex in the direction of the $\R$ factor. Hence $\Om$ is a union
of coaxial balls $\{ x\}\times\overline{B}\big( y',u(x)\big)$. Notice that if $u(x)$ achieves the value~$r$
then necessarily $y'=y$. If $\max u < r$, then $y$ and $y'$ may be different.

\vspace{3mm}

We finally make a comment about the function $u$. Denote by $\Om$ the region bounded by a soap bubble, or an isoperimetric region in $X\times\R$, or an isoperimetric region in $X\times\overline{B}(y,r)$. The interior $U$ of $\Om$ is an open set in $X\times\R$, thus its image under the projection to the $X$ factor is an open set $A\subseteq X$. For each $x\in A$ the value $u(x)$ must be
positive, because the intersection of $U$ with the slice $\{ x\}\times\R$ must be a non-empty open ball. Therefore $U$ is the union of the non-empty open balls $\{ x\}\times B\big( y,u(x)\big)$ as $x$ ranges over~$A$. Such a union of balls is open in $X\times\R$ if and only if $u$  is lower semicontinuous. We are going to see in the next section that $u$ is actually much more regular.


\section{regularity}\label{sec-regularity}

Let $u$ and $A$ be as described in Section~\ref{sec-symmetry}. In
this section we prove that $u$  is ${\mathcal C}^1$ in $A$
and that it is smooth in $A\cap\{ u<r\}$. We also show that
$u$ is continuous in its whole domain (the closure of $A$)
and vanishes on the frontier $\overline{A}\setminus A$.

Let $S$ be a soap bubble or an isoperimetric boundary, in all of $X\times\R$ or in $X\times\overline{B}(y_0,r)$.
We may assume without loss of generality that $y_0=0$ and that $X\times\{ 0\}$ is the symmetry axis of~$S$.

Define a function  $\rho :X\times\R\to{\mathbb R}$ as
follows:
\[ \rho (x,y)=\mbox{\rm dist}\, \big(\, (x,y)\, ,\, X\times\{ 0\}\,\big) =\sqrt{y_1^2+\cdots +y_n^2}\; .\] This function is
smooth away from $X\times\{ 0\}$.

The singular set of $S$, if non-empty, is compact and projects to a compact set in $X$ whose codimension in $X$ is at least~7. The reason for this is that, due to the invariance of $S$ under rotations of the $\R$ factor, any image in $X$ of a singular point comes
from a whole $S^{n-1}$--worth of singular points on~$S$.

Denote by $S_0$ the regular part of $S\cap \{ \rho >0\}$. In $S_0$ there is defined an outer unit normal $\nu$. We denote by $\rho_\nu$ the derivative of $\rho$ along this normal.

Let $\pi :S\to X$ be the restriction of the projection $X\times\R\to X$. Let $\Om$ be the region bounded by $S$. The interior
of $\Om$ is the union of the coaxial balls $\{x\}\times B\big(0,u(x)\big)$ for some
function $u:A\to{\mathbb R}^+$ that may be described
as $u=\rho\circ\pi^{-1}$. At a point  $z\in S_0$ where $\rho_\nu\neq 0$,
the map $\pi$ is a submersion; thus $u$ is near $\pi (z)$ as regular as $S_0$ is near $z$ (smooth or ${\mathcal C}^1$, depending on the case).
On the other hand, if $\rho_\nu (z)=0$ then the gradient of $u$ is infinite at~$\pi (z)$. In order to prove the regularity part of Theorem~\ref{sym-smooth}, we study the vanishing of~$\rho_\nu$.

\begin{lem}\label{lemma-H.}
Let $M$ be a Riemannian manifold and $S\subset M$ a hypersurface. Let $\xi$ be a vector field on $M$ and let $\varphi_t$ be the flow of $\xi$. Denote by $H$ the mean curvature function of $S$ and by $H_t$ the same for $\varphi_t(S)$. Write $\nu$ for the unit normal of $S$ and decompose $\xi =\xi^\top+f\,\nu$. Then for each $p\in S$ we have:
\begin{equation}\label{formula-H.}\left.\frac{\partial}{\partial t}\right|_{t=0}\, H_t\big(\varphi_t(p)\big)\; =\;
\xi^\top_p H-\big(\, \mbox{\rm Ric}(\nu ,\nu )+| II|^2\,\big)_p\, f(p)-(\Delta^S f)_p\; ,\end{equation}
where $\,\mbox{\rm Ric}\,$ is the Ricci tensor of $M$ and $II$ is the second fundamental form of~$S$.
\end{lem}

If $S$ is an isoperimetric boundary in $X\times\overline{B}(0,r)$, let $S_\infty$ denote $S\cap\{ \rho =r\}$. If $S$ is a soap bubble or an isoperimetric boundary in $X\times\R$, let $S_\infty$ be just the empty set.

\begin{prop}
In $S_0\setminus S_\infty$ the following identity holds:
\begin{equation}\label{Schrodinger}
\Delta^S(\rho_\nu )=\left(\frac{n-1}{\rho^2}-\mbox{\rm Ric}\, (\nu
,\nu )-|II|^2\right)\,\rho_\nu \; ,
\end{equation}
\end{prop}

\begin{proof}
Both sides of equality (\ref{Schrodinger}) are zero on the interior of the set $S_0\cap\{\rho_\nu =0\}$. Also, both sides are smooth everywhere on $S_0\setminus S_\infty$. Therefore if we prove the identity on $(S_0\setminus S_\infty )\cap\{\rho_\nu\neq 0\}$ it will also be true on the frontier points of $S_0\cap\{\rho_\nu =0\}$ because these are limits of points where $\rho_\nu\neq 0$.

Let $z_0=(x_0,y)\in S_0\setminus S_\infty$ be a point with $\rho_\nu (z_0)\neq 0$. The function $u$ is smooth in some neighborhood $U^{x_0}$ of $x_0$ in $X$, and
a neighborhood $S^{z_0}$ of $z_0$ in $S$ is described as a rotated graph: $\{\rho =u(x)\, ,\,  x\in U^{x_0}\}$. For any such rotated graph we have:
\begin{equation}\label{H-u}
H=\frac{n-1}{u\,\sqrt{1+|\nabla u|^2}}\;
-\;\mbox{\rm div}\,\frac{\nabla u}{\sqrt{1+|\nabla u|^2}}
\; , \end{equation}
where gradient and  divergence are taken in $X$. In particular, the  hypersurfaces
$S_t=\{\rho =t+u(x)\, ,\, x\in U^{x_0}\}$ have the following mean curvatures:

\begin{eqnarray*}
H_t
&=&\frac{n-1}{(t+u)\,\sqrt{1+|\nabla u|^2}}\;
-\;\mbox{\rm div}\,\frac{\nabla u}{\sqrt{1+|\nabla u|^2}} \;=\\
&=& H_0-\frac{n-1}{u^2\,\sqrt{1+|\nabla u|^2}}\; t\; +\; O(t^2)\;
=\; H_0-\frac{n-1}{\rho ^2}\,\rho_\nu\, t\; +\; O(t^2)
\; .\end{eqnarray*}
But $S_t$  is the image of $S^{z_0}$ under the flow of
$\nabla\rho=(\nabla\rho )^\top +\rho_\nu\,\nu$; thus Lemma~\ref{lemma-H.}
gives:
\[ -\frac{n-1}{\rho ^2}\,\rho_\nu =(\nabla\rho )^\top\, H_0
-\big(\,\mbox{\rm Ric}\, (\nu ,\nu )+|II|^2\,\big)\,\rho_\nu
-\Delta^S(\rho_\nu )\; .\] In addition \ $H_0$ \ is constant,
hence $(\nabla\rho )^\top\, H_0\equiv 0$. Thus
(\ref{Schrodinger}) holds where $\rho_\nu\neq 0$.
\end{proof}

Now (\ref{Schrodinger}) is a Schr\"odinger equation.  By the results in \cite{CG}, a non-negative solution to such equation
that vanishes at a point must vanish everywhere. Notice that the object $S$ is, by our hypotheses, the boundary of a
region $\Om$ that is a union of coaxial balls $\{ x\}\times B\big( 0,u(x)\big)$. Since $\nu$
points outward with respect to $\Om$,
it is $\rho_\nu\geq 0$ everywhere on~$S_0$. We
conclude that on each connected
component of $S_0\setminus S_\infty$ we have either $\rho_\nu\equiv 0$ or $\rho_\nu >0$.

If $S$ is a soap bubble or an isoperimetric boundary in $X\times\R$, then we cannot have $\rho_\nu\equiv 0$ on a connected component of $S_0$, because then a whole orbit of $\nabla\rho$ would be part of $S$ and this contradicts boundedness. In fact, some orbits may be interrupted by the singular points of $S$ but not all orbits, due to the large codimension of the singular set.

If $S$ is an isoperimetric boundary in $X\times\overline{B}(0,r)$ and $S_0\setminus S_\infty$ has a connected component
with $\rho_\nu\equiv 0$, then such component must reach the obstacle $X\times S(0,r)$ and not be tangent to it, which
contradicts the standard regularity for obstacle problems~\cite{Gonzalez}.

We conclude that $\rho_\nu >0$ on all of $S_0\setminus S_\infty$. If $S$ is a soap bubble, this already implies that $u$
is smooth on all of~$A$.

If $S$ is isoperimetric in
$X\times\overline{B}(0,r)$, we let
$A'=\{ x\in A\, |\, u(x)<r\}$; in this case
$S\cap\{\rho=r\}$ is compact
and $A'$ is open because it equals
$A\setminus\pi (S\cap\{\rho=r\})$. If $S$ is isoperimetric
on $X\times\R$, we let $A' =A$.
There is a subset $K\subset A'$, of codimension at least~7 in $A'$, such that $u$ is smooth in $A'\setminus K$. In addition $K$ is locally closed, because it is the intersection of the open set $A'$ with the (compact) image under $\pi$ of the singular set.

We can switch from a rotated graph description to a Cartesian graph description. If $(y_1,\dots ,y_n)$ are orthonormal coordinates in $\R$, then $S_0\setminus (S_\infty\cup\{ y_n=0\})$ is the union of two graphs $\{ y_n=\pm f(x,y_1,\dots y_{n-1})\}$, where $f$ is the following function:
\begin{equation}\label{f-graph}
f(x,y_1,\dots y_{n-1})\; =\; \sqrt{u(x)^2-y_1^2-\cdots -y_{n-1}^2}
\; , \end{equation}
defined on the open subset $U=\{ y_1^2+\cdots +y_{n-1}^2<u(x)^2\}\subset A'\times{\mathbb R}^{n-1}$. It is clear
that $A'\times\{ 0\}\subset U$. What we have proved so far implies that $f$ is smooth except perhaps on a locally closed set whose codimension in $X\times{\mathbb R}^{n-1}$ is at least~7.

De Giorgi and Stampacchia have a theorem \cite{dGS} which says
that if $f$ is a $C^2$ function defined on $U\setminus K$, with
$U\subset{\mathbb R}^d$ open and  $K\subset U$  compact of
codimension greater than~1, and if the graph of $f$ is minimal
in ${\mathbb R}^{d+1}$, then $f$ extends to a $C^2$ function on
all of $U$. L. Simon \cite{Simon} has improved this theorem, so that we only need $K$ to be
locally closed in $U$ and the graph of $f$ over $U\setminus K$ may satisfy a
PDE of some general type which includes the case of constant mean
curvature in $U\times{\mathbb R}$, with the factor $U$
having any smooth Riemann metric.

We can apply L. Simon's theorem to the function $f|_{U\setminus K}$ defined by (\ref{f-graph}). It thus extends
to a function $\widetilde{f}$ which is smooth in all of $U$. We know from Section~\ref{sec-symmetry} that the function $u$ is
lower semicontinuous; this alone does not force it to coincide with $\widetilde{f}(x,0)$, but it they did not coincide then
the region $\Om$ would not meet its boundary in the nice way that isoperimetric regions do, see e.g.~\cite{Maggi}. Finally we have proved that $u$ is smooth on all of~$A'$. In the case $A'\neq A$, at least we have $u$ of class ${\mathcal C}^1$ on
all of $A$ because $\rho_\nu =1$ everywhere on $S\cap\{\rho =r\}$.

We consider now what happens at the points $z_0=(x_0,0)\in S\cap (X\times\{ 0\})$, where the object $S$ meets its symmetry axis. Suppose there is a sequence $\{ z_j\}\subset S_0$ that converges to $z_0$ and satisfies $\lim\rho (z_j)=\delta >0$. Then the
Euclidean ball $\{ x_0\}\times B(0,\delta )$ is entirely contained in $S$ and we have $\rho_\nu =0$ at points on $S$ near the axis but not on the axis. Such points belong to $S_0\setminus S_\infty$ and we have a contradiction. Therefore, along any sequence
converging to $z_0$ we have $\liminf \rho (z_j)=0$. Since $u$ is lower semicontinuous and non-negative, we conclude that $u$
extends continuously from $A$ to the closure $\overline{A}$ and its value on the frontier $\overline{A}\setminus A$ is identically zero.

The proof of Theorem~\ref{sym-smooth} is now complete.


\section{Estimates for isoperimetric boundaries}\label{sec-estimates-1}

In this section we prove the isoperimetric boundary part of  Theorem \ref{teo-estimates}.

Recall from Section \ref{sec-regularity} that
$\rho :X\times\R\to{\mathbb R}$ is given by
 $\rho (x,y)=|y|$. Again $S=\partial\Om$ will be a soap bubble or an
 isoperimetric boundary, in all of $X\times\R$ or in
$X\times\overline{B}(0,r)$. We refer to the last possibility as
\textit{the obstacle case,} because is such case
$X\times S(0,r)$ acts as an obstacle that the isoperimetric
region $\Om$ may hit. In the three cases we assume $X\times\{ 0\}$
to be the symmetry axis of $\Om$ and $S$.

We now know that $S\cap\{\rho >0\}$ is the rotated graph of a differentiable
function $u:A\to{\mathbb R}^+$. In the obstacle case $u$ is
${\mathcal C}^1$ in all of $A$ and smooth in $\{ u<r\}$. In the
other two cases $u$ is smooth on all of~$A$.

Let $k=\dim X$ and $d=k+n=\dim (X\times\R )$. We shall write ``area'' for $(d-1)$-dimensional Hausdorff measure.

A preliminary estimate (\ref{H-ineq}) comes from computing in two different
ways the first variation of area of pieces of $S$  under the flow of
$\xi =y_1\,\partial_{y_1}+\cdots
+y_n\,\partial_{y_n}$, a vector  field parallel to the
$\R$ factor and vanishing along the axis $X\times\{ 0\}$.

Given a tiny element of hypersurface with normal unit vector~${\bf
v}\,$, \ the flow of a given vector field $\xi$ modifies the area
of such piece at the rate  $\mbox{\rm div}\,\xi -{\bf
v}\cdot\nabla_{\bf v}\xi\,$, and so the first variation of area
is $\int_{\Sigma}\big( \mbox{\rm div}\,\xi
-\nu\cdot\nabla_\nu\xi\big)$  for any hypersurface $\Sigma\,$,
compact or non--compact.

We first consider the obstacle case. It will be trival to adapt the argument to the other two cases.

The hypersurface  $S\cap\{\rho >0\}$  is transverse to almost all level sets of
 $\rho\,$.  Hence $S(\varepsilon )=S\cap\{ \varepsilon \leq\rho
\leq r-\varepsilon\}$  is a compact hypersurface with smooth boundary
for almost every  $\varepsilon\in (0,r/2)\,$.  Denoting by
$\eta_{\varepsilon}$  the outer conormal of  $\partial
S(\varepsilon )$ within $S$,  and applying formula
(\ref{def-H}) of Section~\ref{sec-monotonicity}
to  $S(\varepsilon )$,  the result is:
\[
\int_{S(\varepsilon )}\big(
\mbox{\rm div}\,\xi -\nu\cdot\nabla_\nu\xi\big)\;
=\;\int_{S(\varepsilon )} H\,\nu\cdot\xi +\int_{S\cap\{\rho
=\varepsilon\}} \xi\cdot\eta_{\varepsilon} + \int_{S\cap\{\rho
=r-\varepsilon\}}\xi\cdot\eta_{\varepsilon}
\; ,\]
where the last
term is non--negative because  $S$  is a rotated graph  $\{\rho
=u(x)\}$  and $\xi \equiv \rho\,\nabla\rho$.  We thus have:
\begin{equation}\label{var}
\int_{S(\varepsilon )}\big(
\mbox{\rm div}\,\xi -\nu\cdot\nabla_\nu\xi\big) \;\geq\;
\int_{S(\varepsilon )} H\,\nu\cdot\xi \; +\; \int_{S\cap\{\rho
=\varepsilon\}} \xi\cdot\eta_{\varepsilon}
\; .\end{equation}
The vector field  $\xi$ satisfies\ $|\xi |=\varepsilon$ along
$\{\rho =\varepsilon\}$  and the last term in (\ref{var}) is
bounded in absolute value by $\varepsilon\cdot {\mathcal H}^{d-2}\big(S\cap\{\rho
=\varepsilon\}\big)$. For each $t\in (0,r)$ define
$\delta (t) =\inf\limits_{0<\varepsilon
<t}\varepsilon\cdot{\mathcal H}^{d-2}\big(S\cap\{\rho
=\varepsilon\}\big)$. The coarea formula gives:
\[
\infty \; >\;\mbox{\rm area}\,\big(\, S\cap\{ 0<\rho
<t\}\,\big)\;\geq\;\int_0^t {\mathcal H}^{d-2}\big(S\cap\{\rho
=\varepsilon\}\big) \; d\,\varepsilon \;\geq\;\int_0^t\frac{\delta
(t)}{\varepsilon}\; d\,\varepsilon
\; ,\]
which implies  $\delta (t) =0$ for all $t$. We deduce the existence
of a sequence $\varepsilon_j\to 0$
such that $\int_{S\cap\{\rho =\varepsilon_j\}}
\xi\cdot\eta_{\varepsilon_j}\to 0$. In the limit as
$j\to\infty$, inequality (\ref{var}) thus becomes:
\begin{equation}\label{div-H}
\int_{S\cap\{ 0<\rho <r\}}\big( \mbox{\rm div}\,\xi -\nu\cdot\nabla_\nu\xi\big)
\; \geq \; \int_{S\cap\{ 0<\rho <r\}} H\,\nu\cdot\xi
\; . \end{equation}

\begin{defn}
For each $t\in [0,r]$, define \ $X_t=\{\, x\in X\; |\; u(x)\geq t\,\}$,
which is a region with smooth boundary for almost every
$t$.

The \textbf{thick part} of an isoperimetric region $\Om$ in $X\times\overline{B}(0,r)$ is $\Om_{\mbox{\rm\scriptsize thick}} =X_r\times B(0,r)$. See Figure~\ref{Othick}.
\end{defn}

\begin{figure}[ht]
\includegraphics[scale=0.3]{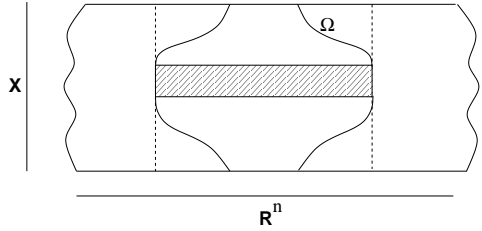}
\caption{The thick part of $\Om$}
\label{Othick}
\end{figure}

For any vector $\bf v$ tangent to $X\times\R$, the
derivative $\nabla_{\bf v}\xi$ is the orthogonal component of $\bf
v$ in the direction of the $\R$ factor.  In particular
$0\leq\nu\cdot\nabla_\nu\xi\leq 1$. It follows that:
\begin{equation}\label{div-ineq}
n-1\;\leq\; \mbox{\rm div}\,\xi -\nu\cdot\nabla_\nu\xi\;\leq\; n
\; .\end{equation}
Now (\ref{div-H}) and the second inequality in (\ref{div-ineq}) give:
\begin{equation}\label{variation}
n\cdot\mbox{\rm area}\, (S)\;\geq\;
\int_{S\cap\{ 0<\rho <r\}}\big( \mbox{\rm div}\,\xi -\nu\cdot\nabla_\nu\xi\big) \;\geq\;
\int_{S\cap\{ 0<\rho <r\}} H\, \nu\cdot\xi
\; . \end{equation}
We assume that $X_r$ is a smooth region.   If it is not, we can
reach the same results by passing to a limit as $t\nearrow r$. The outer normal $\overline{\nu}$
of $\Om \setminus\Om_{\mbox{\rm\scriptsize thick}}$ coincides with $\nu$
 along $S\cap\{ 0<\rho <r\}$ and with $-\nu_r$ along $(\partial X_r)\times B(0,r)$, where
$\nu_r$ is the obvious lift of the outer unit normal of $X_r$ in $X$. Since $\xi\cdot\nu_r\equiv 0$, we have:
\[
\int_{S\cap\{ 0<\rho <r\}} \nu\cdot\xi =
\int_{  \partial (\Om \setminus\Om_{\mbox{\rm\scriptsize thick}})  } \overline{\nu}\cdot\xi
\; .\]
Using the divergence theorem, and  writing $v$ for the volume of $\Om$,  we get:
\[
\int_{S\cap\{ 0<\rho <r\}} H\nu\cdot\xi
\; =\; H\,\int_{\partial (\Om\setminus\Om_{\mbox{\rm\scriptsize thick}})}
\overline{\nu}\cdot\xi
\; =\; n\, H\cdot \big(\, v-{\mathcal H}^d(\Om_{\mbox{\rm\scriptsize thick}})\,\big)
\; ,\]
and this together with inequality (\ref{variation}) yields
our preliminary estimate for the mean curvature:
\begin{equation}\label{H-ineq}
 \big(\, v-{\mathcal H}^d(\Om_{\mbox{\rm\scriptsize thick}})\,\big)\, H\leq\mbox{\rm area}\,(S)
\; .\end{equation}
This is for the obstacle case, but if $S$ is a soap bubble or is isoperimetric in all of $X\times\R$ then $\Om_{\mbox{\rm\scriptsize thick}}$ is empty and (\ref{H-ineq}) simplifies to:
\begin{equation}\label{H-ineq-0}
v H\leq\mbox{\rm area}\,(S)
\; . \end{equation}
We now prove the estimate (\ref{isop-H}) in Theorem~\ref{teo-estimates}. When $\Om$ is
an isoperimetric region in all of $X\times\R$, it has less boundary area than
the cylinder $X\times (\mbox{\rm ball})$ of the same volume, that is:
\begin{equation}\label{comp-cyl}
\mbox{\rm area}(S )\leq n\, (\omega_n|X|)^{1/n}\cdot v^{\frac{n-1}{n}}
\; , \end{equation}
and this, combined with inequality (\ref{H-ineq-0}), gives (\ref{isop-H}).

Inequality (\ref{b-isop-H}) in  Theorem~\ref{teo-estimates}
follows in exactly the same way if
$ v-{\mathcal H}^d(\Om_{\mbox{\rm\scriptsize thick}})$
is greater than $v/2$ for $r$ sufficiently large. The next lemma
thus finishes the proof of~(\ref{b-isop-H}).

\begin{lem}\label{halfvolume}
For sufficiently large $r$, depending on $v$, the part
$\Om_{\mbox{\rm\scriptsize thick}}$ contains less than
half the volume of $\Om$.
\end{lem}

\begin{proof}
We assume that  $v/2\leq {\mathcal H}^d(\Om_{\mbox{\rm\scriptsize thick}})$ and derive an upper
bound for $r$.

\vspace{3mm}

The hypothesis $v/2\leq {\mathcal H}^d(\Om_{\mbox{\rm\scriptsize thick}})$  is equivalent to
$\frac{1}{2\,\omega_n}\, v\, r^{-n}\;\leq\; {\mathcal H}^k (X_r)\,$.

Notice that ${\mathcal H}^k(X_t)$ is a decreasing function of~$t$. In
particular, for any $s\in \left[ 0,\frac{r}{2}\right]$  and any
$t \in \left[\frac{r}{2}\, ,\, r\right]$ we have
${\mathcal H}^k(X_s) \geq {\mathcal H}^k (X_{t})\,$,  whence:
\[
v\; =\; {\mathcal H}^d\, (\Om )\; =\; \int_0^r n\,\omega_n\, s^{n-1}\,{\mathcal H}^k (X_s)\,
ds\; \geq\; \omega_n\,{\mathcal H}^k(X_{t})\,\int_0^{r/2}n\,
s^{n-1}\, ds
\; . \]
We then have, for all $t\in \left[\frac{r}{2}\, ,\, r\right]$, the following inequalities:
\begin{equation}\label{2-sided}
{\textstyle\frac{1}{2\omega_n}}\; v\, r^{-n} \;\leq\; {\mathcal H}^k (X_r)
\;\leq\;{\mathcal H}^k(X_{t}) \;\leq\;
{\textstyle\frac{2^n}{\omega_n}}\; v\, r^{-n}
\; . \end{equation}
Make now the extra hypothesis
$r^n\geq\frac{2^{n+1}}{\omega_n}\;\frac{v}{{\mathcal H}^k(X)}$,  which by
the last inequality in (\ref{2-sided}) ensures:
\begin{equation}\label{half-X}
{\mathcal H}^k(X_t)\leq \displaystyle\frac{1}{2}\,{\mathcal H}^k
(X)\quad \mbox{\rm for }\; t\in {\textstyle\left[\frac{r}{2}\, ,\,
r\right]}
\; .\end{equation}
Let $c_X$ be the \textbf{isoperimetric constant} of $X$, so that for every region $Y\subset X$ we have:
\[
{\mathcal H}^{k-1}(\partial Y)\geq
c_X\cdot\max\big(\,{\mathcal H}^k(Y)\, ,\,{\mathcal H}^k(X\setminus Y) \,\big)^{\frac{k-1}{k}}
\; . \]
Applying this inequality, (\ref{half-X}), and (\ref{2-sided}) to those $X_t$ which are smooth, we
get for almost every $t\in \left[\frac{r}{2}\, ,\, r\right]$:
\[
{\mathcal H}^{k-1}(\partial X_t)\;\geq\; c_X\cdot
\big({\mathcal H}^k\, (X_t)\big)^{\frac{k-1}{k}}\;\geq\;
\mbox{\rm const}\cdot\big( v\cdot r^{-n}\big)^{\frac{k-1}{k}}
\; .\]
The coarea formula gives in turn the following estimate:
\[
\mbox{\rm area}\, (S)\; >\int_{r/2}^r n\,\omega_n\,
t^{n-1}\,{\mathcal H}^{k-1}\, (\partial X_t)\,  dt
\;\geq\;\mbox{\rm const}\cdot r^n\cdot (v\, r^{-n})^{\frac{k-1}{k}}
\; ,\]
which implies:
\[
r^n\;\leq\;\mbox{\rm const}\cdot v^{1-k}\cdot\big(\,\mbox{\rm area}
(S)\,\big)^{k}\;\leq\;\mbox{\rm const}\cdot v^{1-(k/n)}
\; ,\]
the last inequality coming from (\ref{comp-cyl}).

\vspace{3mm}

From the hypothesis
$v/2\leq {\mathcal H}^d(\Om_{\mbox{\rm\scriptsize thick}})$
we have deduced that either
$r^n<\frac{2^{n+1}}{\omega_n}\;\frac{v}{{\mathcal H}^k(X)}$ \ or \ $r^n\leq\mbox{\rm const}\cdot v^{1-(k/n)}\,$. These are two upper bounds for $r$ which depend only on~$v$.
Hence, for $r$ larger than these bounds it must be \ ${\mathcal H}^d(\Om_{\mbox{\rm\scriptsize thick}}) <v/2$.

This proves Lemma~\ref{halfvolume} and, as we have explained, estimate (\ref{b-isop-H}) in Theorem~\ref{teo-estimates}.
\end{proof}

We shall now prove estimate (\ref{H-osc-bound}) for an isoperimetric
boundary $S$  in $X\times\overline{B}(0,r)$. Again we assume $X\times\{ 0\}$ is the symmetry axis of~$S$.

We know that $S$ has constant mean curvature in $\{ 0<\rho <r\}$. Assuming the volume $v$ enclosed by $S$ to be
larger than some arbitrary value $v_0$, and $r$ large enough for $v$, the estimate (\ref{b-isop-H}) provides an upper bound $H_0$ for said constant mean curvature. Taking this mean curvature bound to the monotonicity inequality (\ref{area-estimate}), we obtain constants $\varepsilon$ and $\delta$ such that:
\[
\mbox{\rm if }\; z_0\in S\cap\{\varepsilon <\rho <r-\varepsilon\}\;\;\;\;\mbox{\rm then }\;\;\;
\mbox{\rm area}\big(\, S\cap  B^{X\times\R}(z_0,\varepsilon )\,\big)\geq\delta
\; .\]
Define:
\[ \rho_0=\min (\rho|_S)\quad ,\quad \rho_1=\max (\rho|_S)\; .\]
The projection $X\times\R\to\R$ maps $S$ onto a Euclidean ring with radii $\rho_0$ and $\rho_1$. Euclidean space $\R$ has
a \textbf{packing constant} $C(n)$ such that a ring with those radii can pack $\ell$ disjoint Euclidean balls of radius $\varepsilon$,
where $\ell\geq C(n)\cdot (\rho_1^n-\rho_0^n)/\varepsilon^n$. These balls lift to disjoint distance balls in $X\times\R$, centered at points of $S\cap\{\varepsilon <\rho <r-\varepsilon\}$. It follows that:
\begin{equation}\label{area-rho}
\mbox{\rm area} (S)\geq \mbox{\rm const}\cdot (\rho_1^n-\rho_0^n)
\; , \end{equation}
the constant depending only on $n$ and $X$. Comparing $S$ with a cylinder $X\times (\mbox{\rm sphere})$ that encloses the same volume, we get $\,\mbox{\rm area}(S)\leq\mbox{\rm const}\cdot\rho_1^{n-1}$, therefore:
\begin{equation}\label{powers}
\rho_1^n-\rho_0^n\leq\mbox{\rm const}\cdot\rho_1^{n-1}
\; ,\end{equation}
with the constant depending only on $n$ and $X$. Observing that:
\[
\rho_1^n-\rho_0^n\; =\; (\rho_1-\rho_0)\cdot (\rho_1^{n-1}+\rho_1^{n-2}\rho_0+\cdots +\rho_0^{n-1})
\;\geq\; (\rho_1-\rho_0)\cdot\rho_1^{n-1}
\; ,\]
we deduce from (\ref{powers}) the inequality $(\rho_1-\rho_0)\cdot\rho_1^{n-1}\leq\mbox{\rm const}\cdot\rho_1^{n-1}$,
which is equivalent to the first inequality in~(\ref{H-osc-bound}). This estimate makes $v$ comparable to $\rho_1^n$, thus the
second inequality  in~(\ref{H-osc-bound}) follows from~(\ref{b-isop-H}).

\begin{rem}
The constant in (\ref{powers}) really depends on $n$, $X$, and the chosen value $v_0$. This means that we have an
estimate valid also for small enclosed volumes:
\begin{equation}\label{osc-bound}
\rho_1-\rho_0\;\leq\;\mbox{\rm const}_v
\, , \end{equation}
where the constant depends on $v$ but not on the radius $r$ of the domain $X\times\overline{B}(0,r)$ where $S$ is isoperimetric.
\end{rem}

We also have proved now estimate (\ref{H-osc-bound})  for
regions isoperimetric in $X\times\R$ of volume larger than $v_0$,
because they are isoperimetric in any compact domain.


\section{Estimates for soap bubbles}\label{sec-estimates-2}

In this section we prove the soap bubble part of Theorem \ref{teo-estimates}.

We start with the upper bounds for mean curvature: one under
the hypothesis $\,\mbox{\rm Ric}\geq 0$, 
the other under no special hypothesis. These bounds for $H$
will in turn allow us to get the radius bounds. The idea 
for the mean curvature estimates is that large $H$ would force
$S$ to ``roll up'' and bound a region of small volume. The first
result along these lines was obtained by J. Serrin
\cite[pages~85-87]{Serrin} for surfaces in ${\mathbb R}^3$.
Serrin uses a formula of G.~Darboux for parallel surfaces,  and
deals with the possible singularity of a parallel surface at a focal
point. W. Meeks has a similar result in
\cite[page~544]{Meeks} for hypersurfaces in $\R$.
His calculation is equivalent to that of Serrin, but he works in
the hypersurface instead of its parallel image and 
the singularities do not show up. He considers a
height function $x_n$  and the corresponding component \
$\nu_n$ of the unit normal, then observes that (with our
convention for $H$) the function 
$x_n-\frac{n-1}{H}\,\nu_n$  is subharmonic in the
hypersurface.  If $x_n|_{\partial S}\equiv 0$,  then
$x_n\leq \frac{n-1}{H}$  by the maximum principle.  Intuitively,
if the hypersurface is strongly curved then it cannot reach far
out in a given direction.

We are going to imitate that argument here. We shall multiply a
component of $\nu$ by a constant parameter $p $, then
we choose suitable values for this parameter.

For any function $\varphi$ on a domain of $X\times\R$, define the tangential Laplacian as follows:
\begin{equation}\label{def-tangLap}
\Delta^\top \varphi =\sum_{j=1}^{d-1}\mbox{\rm Hess}(\varphi )({\bf e}_j,{\bf e}_j)
\; ,\end{equation}
where ${\bf e}_1,\dots ,{\bf e}_{d-1}$ is an orthonormal basis of $TS$. With our convention
for $H$, the following holds:
\[
\Delta^S\rho =\Delta^\top\rho -H\,\rho_\nu\geq - H\,\rho_\nu
\; ,\]
because the Hessian of $\rho$ is positive semidefinite. This and
formula (\ref{Schrodinger}) of Section~\ref{sec-regularity} imply
that for any constant $p $ we have:
\begin{equation}\label{Ric-ineq}
\Delta^S\, (\rho -p \,\rho_\nu )\;\geq\;
\left(\, p \, |II|^2+p \,\mbox{\rm Ric}\,
(\nu ,\nu )-p \,\frac{n-1}{\rho^2}-H\,\right)\,\rho_\nu
\; .\end{equation}
Restrict to $p  >0$, and recall that $\rho_\nu >0$. Introduce now the
 hypothesis $\,\mbox{\rm Ric}\geq 0$, then (\ref{Ric-ineq}) simplifies to:
\[
\Delta^S\, (\rho -p \,\rho_\nu )\;\geq\;
 \left(\, p \, |II|^2-H-p \,\frac{n-1}{\rho^2}\,\right)\,\rho_\nu
\; ,\]
and Newton's inequality  $|II|^2\geq H^2/(d-1)$   leads to:
\[
\Delta^S\, (\rho -p \,\rho_\nu )\;\geq\;
\left(\,p \,\frac{H^2}{d-1}-H-p \frac{n-1}{\rho^2}\,\right)\,\rho_\nu
\; .\]

Again let $\rho_1,\rho_0$ be the maximum and minimum,
respectively, of $\rho$ over $S$. The useful choice here is $p  =\rho_1/4$, then:
\[
\Delta^S\, \left(\rho -\frac{\rho_1}{4}\,\rho_\nu \right)\;\geq\;
\left(\,\frac{H^2}{4\, (d-1)}\,\rho_1-H-\frac{n-1}{4\,\rho^2}\,\rho_1\,\right)\,\rho_\nu
\; .\]
Consider the hypersurface piece $\Sigma =S\cap\{\rho\geq\rho_1/2\}$. The boundary $\partial\Sigma$
may be empty, unless $\rho_0<\rho_1/2$. At all points of $\partial\Sigma$ (if any) we have
$\rho -(\rho_1/4)\,\rho_\nu \leq \rho_1/2$  while
there are points on the interior of $\Sigma$ where
$\rho -(\rho_1/4)\,\rho_\nu >\rho_1/2$. For example,
a point $z\in S$ where $\rho(z)=\rho_1$ is interior to $\Sigma$ and gives
$\big(\rho -(\rho_1/4)\,\rho_\nu\big)_z=(3/4)\,\rho_1$. Therefore the maximum of
$\rho -(\rho_1/4)\,\rho_\nu $ over $\Sigma$ is achieved at an
interior point $z_0$. At $z_0$ we have $4\,\rho^2>\rho_1^2$ and
$\Delta^S\big(\rho -(\rho_1/4)\,\rho_\nu\big)\leq 0$, hence:
\begin{equation}\label{1}
0\geq \left(\,\frac{H^2}{4\, (d-1)}\,\rho_1-H-\frac{n-1}{4\,\rho^2}\,\rho_1\,\right)(z_0)
\geq \frac{H^2}{4\, (d-1)}\,\rho_1-H-\frac{n-1}{\rho_1}
\; .\end{equation}
Suppose $n\geq 2$. The hypersurface $\{\rho=\rho_1\}$ touches $S$ tangentially from outside and
has constant mean curvature $(n-1)/\rho_1$, thus $H\geq (n-1)/\rho_1>0$. Taking this
lower bound for $H$ to (\ref{1}), we obtain:
\[
0\geq \frac{H^2}{4\, (d-1)}\,\rho_1-2\, H
\; ,\]
and since $H>0$ we deduce $H\leq 8\, (d-1)/\rho_1$. If the ambient space is
$X\times{\mathbb R}$, then $n-1=0$ and (\ref{1}) reduces to:
\[
0\geq\frac{H^2}{4\, (d-1)}\,\rho_1 -H
\; ,\]
then $H$ is either 0 or a positive number not greater
than $4\, (d-1)/\rho_1$, in either case $H\leq 4\, (d-1)/\rho_1 = (4/\rho_1)\cdot\dim (X)$.

\vspace{3mm}

We now prove the mean curvature bound when $S$
is a (non-isoperimetric) soap bubble enclosing volume $v$ and the Ricci curvature of $X$ is negative
somewhere. Again we shall have to separate the case $n\geq 2$ from the case $n=1$. 
We introduce the constant:
\begin{equation}\label{def-R0}
R_0=\max_{|{\bf v}|=1}\,
\big(\, -\mbox{\rm Ric}\, ({\bf v},{\bf v})\,\big)^+
\; .\end{equation}
The number defined by (\ref{def-R0}) is the same whether
we consider Ric as the Ricci tensor of $X$ and $\bf v$ ranging
over unit tangent vectors to $X$,
or we consider Ric as the Ricci tensor of
$X\times\R$ and $\bf v$ ranging over unit
tangent vectors to $X\times\R$. Notice that $\,\mbox{\rm Ric}\geq 0$
is equivalent to $R_0=0$.

From (\ref{Ric-ineq}) and Newton's inequality, we now deduce for $p  >0$:
\[
\Delta^S(\rho -p \,\rho_\nu )\geq\left(
p \,\frac{H^2}{d-1}-H-R_0\, p  -p \,\frac{n-1}{\rho^2}
\right)\,\rho_\nu
\; .\]
As pointed out above, if $n\geq 2$ then $H\geq (n-1)/\rho_1>0$. Let us see
that the choice $p  =d/H$ is useful is this situation, leaving
the case $n=1$ for later. First we obtain:
\[
\Delta^S\left(\rho -\frac{d}{H}\,\rho_\nu\right) \;\geq\;
\left(\frac{H}{d-1}-\frac{d\, R_0}{H}-\frac{(n-1)\, d}{H\,\rho^2}\right)\,\rho_\nu
\; .\]
Assume $\rho_1\geq 2$. In particular, the hypersurface
$\Sigma =S\cap\{\rho\geq1\}$ is non-empty. On $\Sigma$ one has:
\begin{equation}\label{eq-H}
\Delta^S\left(\rho -\frac{d}{H}\,\rho_\nu\right) \;\geq\;
\left(\frac{H}{d-1}-\frac{d\, R_0}{H}-\frac{(n-1)\, d}{H}\right)\,\rho_\nu
\; .\end{equation}
The factor multiplying $\rho_\nu$ in the right-hand side is a strictly increasing function of $H>0$,
and it equals 0 for a unique positive value $H_0$ of $H$. Let us see that for large
enclosed volume we have $H\leq H_0$.

Suppose $H>H_0$,  then $\rho -(d/H)\rho_\nu$
is strictly subharmonic on $\Sigma$. This is impossible if $\partial\Sigma$ is empty. If
$\partial\Sigma =S\cap\{\rho =1\}$ is non-empty, then the maximum of
$\rho -(d/H)\rho_\nu$ on $\Sigma$ is reached somewhere on $\partial\Sigma$. It follows that 
$\rho\leq 1+(d/H)<1+(d/H_0)$ on all of $S$, which cannot be true if $S$
encloses a large enough volume. Thus $n\geq 2$ plus  large enclosed volume forces $H\leq H_0$.

Assume now $n=1$. In this case, for each constant $t$ the
hypersurface $\{\rho = t\}$ is minimal. If $\rho_0>0$,
then $S$ is sandwiched between the two minimal hypersurfaces $\{\rho =\rho_0\}$ and $\{ \rho=\rho_1\}$
that touch $S$ tangentially; this implies $0\leq H\leq 0$, thereby forcing $S$ to be minimal and, by the maximum principle, $\rho_1=\rho_0$. Hence $S$ must be of the form $X\times\{ -t,t\}$.

The remaining case is $n=1$ and $\rho_0=0$. We know that $S$ is a
symmetric graph $\{\rho =\pm u(x)\}$, with $x$
ranging over a proper subset $X_0\subset X$. This
is, for instance, the situation for
the surfaces $S_v$ of Theorem~\ref{teo-estables}. The minimal hypersurface
$\{\rho =0\}$ is not tangent to $S$ now, but $\{\rho =\rho_1\}$ still  is.
Hence $H\geq 0$, and in fact it must be $H>0$ by the maximum principle. Then
formula (\ref{eq-H}) is valid again, adopting the following form:
\[
\Delta^S\left(\rho -\frac{d}{H}\,\rho_\nu\right) \;\geq\;
\left(\frac{H}{d-1}-\frac{d\, R_0}{H}\right)\,\rho_\nu
\; .\]
We define $H_0$ by $\frac{H_0}{d-1}-\frac{d\,R_0}{H_0}=0$ and we deduce,
as before, that $H\leq H_0$ for large enclosed volume.

\vspace{3mm}

Having proved the mean curvature bounds in (\ref{H-osc-bound}) and (\ref{H-max-bound}), 
we shall now prove the radius bounds. Since $S$ is supposed to be not too small, we have
$H$ less than some constant; then monotonicity plus a sphere packing argument, as we did in
Section~\ref{sec-estimates-1}, yields
again a lower area bound like (\ref{area-rho}) of Section~\ref{sec-estimates-1}. We need some
upper bound for area in order to arrive at a radius estimate.

Once more we assume $n\geq 2$ and leave
the $n=1$ case for later. Since $S$ is closed and everywhere smooth, we have:
\begin{equation}\label{H-div-0}
\int_S(\mbox{\rm div}\,\xi -\nu\cdot\nabla_\nu\xi )=\int_SH\,\nu\cdot\xi
\; .\end{equation}
In Section~\ref{sec-estimates-1} we used the second inequality
in formula (\ref{div-ineq}); now we use the first inequality in that formula,
together with equality (\ref{H-div-0}) and the divergence 
theorem, to deduce  $\, (n-1)\,\mbox{\rm area}(S)\leq n\, H\, v$, \ 
and we can divide by $n-1\geq 1$, to get:
\begin{equation}\label{area-upper-bound} 
\mbox{\rm area}(S)\;\leq\;\frac{n}{n-1}\, H\, v
\; . \end{equation}
If $\,\mbox{\rm Ric}\geq 0$, then we have $H\leq\mbox{\rm const}/\rho_1$ which transforms
(\ref{area-upper-bound}) into  an inequality:
\[ \mbox{\rm area}(S)\;\leq\;\mbox{\rm const}\cdot\rho_1^{n-1}\; , \]
even though we do not assume $S$ to be isoperimetric. We then obtain the radius oscillation estimate in 
(\ref{H-osc-bound}) exactly as we did in Section~\ref{sec-estimates-1}.

If $\,\mbox{\rm Ric}$ is negative somewhere, then we only have $H\leq\mbox{\rm const}$ and
all we can get is the estimate:
\[ 
\rho_1^n-\rho_0^n\;\leq\;\mbox{\rm const}\cdot\mbox{\rm area}(S)\leq\mbox{\rm const}\cdot v
\; , \]
and we consider the following dichotomy:
\begin{list}{}{}
\item[--] if \ $\rho_1/\rho_0>2\,$, \ then \ $\big(\,
1-(1/2)^n\,\big)\, \rho_1^n \;\leq\;\mbox{\rm const}\cdot v\,$,
\item[--] if \ $\rho_1/\rho_0\leq 2\,$, \ then \ $v\;\geq\;
\omega_n\, |X|\,\rho_0^n\;\geq\;
\omega_n\, |X|\, (1/2)^n\, \rho_1^n\,$.
\end{list}
In either case we deduce $\rho_1\leq\mbox{\rm const}\cdot v^{1/n}$. This
completes the proof of (\ref{H-max-bound}) for $n\geq 2$.

Suppose now that $n=1$. We still have $H$ bounded above by a constant and,
by monotonicity, an inequality:
\begin{equation}\label{area-lineal}
\mbox{\rm area}(S)\geq\mbox{\rm const}\cdot (\rho_1-\rho_0)
\; .\end{equation}
We again need an upper bound for area. 

Formula (\ref{H-u}) of Section~\ref{sec-regularity} now reduces to
$ H=-\mbox{\rm div}\,\big(\,\nabla u/\sqrt{1+|\nabla u|^2}\,\big)$,
involving only the divergence term. We take advantage of this
by doing an integration by parts. Given a value $s>0$, The function $u-s$ vanishes
along the boundary of $X_s=\{\, u\geq s\}$ and so:
\begin{eqnarray*}
&& \frac{v}{2}\, H+|X_0| \;=\; \int_{X_0} (u \, H+1)\,\geq\;\int_{X_s}\big( (u-s)\, H+1\big)\; =
\\ && = \int_{X_s}\left(\,\frac{|\nabla u|^2}
{\sqrt{1+|\nabla u|^2}}
+1\,\right)\,\geq \;\int_{X_s}\sqrt{1+|\nabla u|^2}
\; =
\\ && = \frac{1}{2}\,\mbox{\rm area}\, (S\cap\{ \rho\geq s\} )
\; ,\end{eqnarray*}
for almost every $s>0$. By letting $s\to 0$ we obtain:
\begin{equation}\label{ineq-dim-1}
\mbox{\rm area}\, (S)\; \leq\; v\, H+2\, |X_0|
\; .\end{equation}
If $\,\mbox{\rm Ric}\geq 0$, then $H\leq\mbox{\rm const}/\rho_1$ and (\ref{ineq-dim-1})
becomes $\,\mbox{\rm area}(S)\leq\mbox{\rm const}$. We take this to (\ref{area-lineal})
and get the radius oscillation bound in~(\ref{H-osc-bound}).

If $\,\mbox{\rm Ric}$ is negative somewhere, then $H\leq\mbox{\rm const}$. Now
(\ref{area-lineal}) and (\ref{ineq-dim-1}) only give:
\[ 
\rho_1-\rho_0\;\leq\; \mbox{\rm const}\cdot v +2\, |X_0|
\; ,\]
which yields $\rho_1-\rho_0\leq\mbox{\rm const}\cdot v$ for $v$ not too small. Then we
consider the dichotomy $\rho_/\rho_1 >2$ or $\rho_1/\rho_0\leq 2$, and in either case
arrive at the radius bound in~(\ref{H-max-bound}).


\section{Existence}\label{sec-existence}

In this section we prove the following.

\begin{thm}\label{teo-existence}
In $X\times\R$ there are isoperimetric regions of every volume.
\end{thm}

Fix a value $v>0$ and for each $r$ let $\Om (r)$ be any region
isoperimetric of volume $v$ in $X\times\overline{B}(0,r)$. These regions exist
because the domains $X\times\overline{B}(0,r)$ are compact. Up
to a translation parallel to the $\R$ factor, we may assume
that $X\times\{ 0\}$ is the symmetry axis of all the $\Om (r)$. Recall
inequality (\ref{osc-bound}) from Section~\ref{sec-estimates-1},
valid for $r$ large (depending on $v$) and where the constant
 depends on $v$ but not on $r$. The following calculation:
\[
v\;\geq\; |X|\cdot\omega_n\cdot \rho_0^n\;\geq\;  |X|\cdot\omega_n\cdot (\rho_1-\mbox{\rm const}_v)^n
\; ,\]
provides a bound for $\rho_1-\mbox{\rm const}_v$ that depends on $v$ but not on $r$, thereby providing one
such bound also for $\rho_1$. We have thus found a radius $r(v)$ such that for $r$ large enough  the regions
$\Om (r)$ are contained in $X\times\overline{B}\big( 0,r(v)\big)$.

Choose a radius $r_0$  larger than $r(v)$
and large enough for $v$, and let $\Om_0$ be any isoperimetric
region in $X\times\overline{B}(0,r_0)$ of volume $v$, with
symmetry axis $X\times\{ 0\}$. We claim that $\Om_0$ is isoperimetric in all of $X\times\R$.

We first compare $\Om_0$ with  bounded regions.  If $D$ is any
bounded region in $X\times\R$ of volume $v$, there is an
$r$ such that $D\subset X\times B(0,r)$ and $r$ is large enough for $v$. There is also a
region $\Omega'$ isoperimetric of volume $v$ in $X\times
\overline{B}(0,r)$,  and so  $\mbox{\rm area}\,(\partial
D)\geq\mbox{\rm area}\, (\partial\Om')$. By the above, a translate
of the region $\Om'$ is contained inside $X\times \overline{B}\big( 0,r(v)\big)$, where
$\Om_0$  is also isoperimetric of volume $v$.
Therefore $\mbox{\rm area}\, (\partial\Om_0 )=\mbox{\rm area}\,
(\partial\Om')$, and so  $\mbox{\rm area}\, (\partial
D)\geq\mbox{\rm area}\, (\partial\Om_0 )$.

Let now $D'$ be an unbounded region of volume $v$ and finite boundary area.
As $r$ goes to infinity the volume of
$D'_r:=D'\cap\big( X\times B(0,r)\big)$  approaches the volume of $D'$,
and same for boundary area. For $r$ large choose a little
ball $B$ in $X\times\R$, some distance apart from $D'_r$ and such that
$D'_r\cup B$ has exactly volume $v$. Then
\[
\mbox{\rm area} (\partial \Om_0 ) \; \leq \;
\mbox{\rm area}\big(\partial ( D'_r\cup B ) \big)\; = \;
 \mbox{\rm area} (\partial D'_r) +\mbox{\rm area}(\partial B )
 \; ,\]
and by letting $r\to\infty$ we get
$\mbox{\rm area}\, (\partial D')\geq\mbox{\rm area}\, (\partial \Om)$,
due to $\mbox{\rm area}\,(\partial{\mathcal B})\to 0$. In fact
$\mbox{\rm area}\, (\partial D')>\mbox{\rm area}\, (\partial \Om)$,
because we proved in Section~\ref{sec-monotonicity} that
no unbounded region is isoperimetric in $X\times\R$.


\section{Proof of Theorems \ref{large-isop} and \ref{large-bubble}}\label{sec-main}

Let $S\subset X\times\R$. If $S$ is an isoperimetric boundary,
or if $X$ has $\,\mbox{\rm Ric}\geq 0$ and $S$ is a soap bubble, then for large enclosed volume $v$
we have estimate (\ref{H-osc-bound}) from Theorem~\ref{teo-estimates}; this
implies in particular that $\min u >0$. Thus in these cases the function $u$, whose rotated graph
is $S$, is defined and positive on all of $X$. Indeed, if it were $u:A\to{\mathbb R}^+$ with $A\neq X$
then the frontier of $A$ would be non-empty, and we saw at the end of Section~\ref{sec-regularity}
that $u$ would vanish there.

In view of this, Theorems \ref{large-isop} and \ref{large-bubble} are corollaries of the following
proposition.

\begin{prop}\label{aux}
Fix a constant $C\,$.  Let  $u:X\rightarrow{\mathbb R}^+$
be a smooth function with the oscillation bound $\max\, u-\min\,
u \leq C$ and such that the rotated graph $S=\{\rho
=u(x)\}\subset X\times\R$  has constant mean
curvature.  If $S$ encloses a sufficiently large volume
(depending on $C$), then $u$ must be constant.
\end{prop}

If $n=1$ and $\min u >0$, we have already explained in Section~\ref{sec-estimates-2}
that $S$ must be of the form $X\times\{ -t,t\}$.

In the rest of this section we prove Proposition~\ref{aux} for $n\geq 2$. Instead of the 
radius function $u$ we shall work with the slice volume function:
\[ \sigma :=\left(\frac{u}{n}\right)^n\; .\]
The choice of the
factor $n^{-n}$ is not important, it just makes formulas a bit simpler.

Consider the average $\overline{u^n}=(1/|X|)\,\int_xu^n$. The number $a=n^{-n}\cdot\overline{u^n}$
is the average of $\sigma$, thus:
\[
\sigma\;\equiv\; a+\tau\;\; ,\;\;\mbox{\rm for some function }\;\tau
\;\;\mbox{\rm with }\;\int_X\tau =0
\; .\]
We fix the exponent $\alpha =\frac{n-1}{n}\in (0,1)$. In terms of $\sigma$ we have:
\begin{equation}\label{area-alpha}
\mbox{\rm area}\big(\,\{\rho =u(x)\}\,\big)\; =\; n^n\,\omega_n\,\int_X
\sqrt{\sigma^{2\alpha}+|\nabla\sigma |^2}
\, .\end{equation}
Given a family $\{ u_t\}$ of radius functions, and the corresponding family $\{\sigma_t\}$, we
define  $\dot{\sigma}=\left.\frac{d}{dt}\right|_{t=0}\sigma_t$.
Direct differentiation in (\ref{area-alpha}) gives:
\begin{equation}\label{variation-alpha}
\left.\frac{d}{dt}\right|_{t=0}\mbox{\rm area}\big(\,\{\rho =u_t(x)\}\,\big)\; =\;
n^n\,\omega_n\,\int_X\frac{\alpha\,\sigma^{2\alpha
-1}\,\dot{\sigma}+\nabla\dot{\sigma}\cdot\nabla\sigma}
{\sqrt{\sigma^{2\alpha}+|\nabla\sigma |^2}}
\; . \end{equation}

We consider the particular deformation $S_t=\{\rho =u_t(x)\}$ defined by:
\[
u_t=\big(\,\overline{u^n}+e^t(u^n-\overline{u^n})\,\big)^{1/n}
\; ,\]
which satisfies $S_0=S$, and has $\sigma_t=a+e^t\,\tau$, so that all $S_t$ enclose the same volume. We must therefore have:
\begin{equation}\label{must}
0=\int_X \frac{\alpha\,\sigma^{2\alpha -1}\,\tau+|\nabla\tau |^2}
{\sqrt{\sigma^{2\alpha}+|\nabla\tau |^2}} \;\geq\; \alpha\, \int_X
\frac{\sigma^{2\alpha -1}\,\tau+|\nabla\tau |^2}
{\sqrt{\sigma^{2\alpha}+|\nabla\tau |^2}}
\; .\end{equation}

\begin{lem}\label{alpha-ineq}
There is a positive constant $C''$, depending only on $n$, $X$, and the constant $C$ from
Proposition~\ref{aux}, such that for large enough enclosed volume we have:
\begin{equation}\label{2}
\int_X
\frac{\sigma^{2\alpha -1}\,\tau+|\nabla\tau |^2}
{\sqrt{\sigma^{2\alpha}+|\nabla\tau |^2}} \;\geq\;
a^{-\alpha}\, \int_X\left(\, C''\,  |\nabla\tau |^2
-\frac{4}{n}\, a^{2\alpha -2}\,\tau^2\,\right)
\; .\end{equation}
\end{lem}

Using this lemma, we shall now finish the proof of Proposition~\ref{aux}. Let $\lambda_1(X)$ be the first eigenvalue
of the Laplacian in $X$ and observe that:
\begin{eqnarray}
&& \int_X\left(\, C''\, |\nabla\tau |^2
-\frac{4}{n}\, a^{2\alpha -2}\,\tau^2\,\right)\; =\;
 \int_X\left(\, C''\, |\nabla\tau |^2
-\frac{4}{n}\, a^{-2/n}\,\tau^2\,\right)
\;\geq  \nonumber
\\  \label{3} && \geq \; \left(\, C''-\frac{4}{n}\,\frac{1}{\lambda_1(X)}\,\frac{1}{a^{2/n}}\right)
\,\int_X|\nabla \tau |^2
\; .\end{eqnarray}
The average $a$ becomes arbitrarily large as the enclosed volume increases, and the
coefficient $C''-4/(n\,\lambda_1(X)\,a^{2/n})$ becomes positive. Then the expression (\ref{3}) is positive
unless $\nabla\tau\equiv 0$, which  forces $\tau$ to be zero because it has zero average. Therefore, for
large enclosed volume the integrals
in (\ref{2}) and (\ref{must}) are positive unless
$\sigma$ and $u$ are constants, which means that the equality
in (\ref{must}) only holds true if $S=X\times (\mbox{\rm sphere})$. Proposition~\ref{aux} is now proved, and
also Theorems \ref{large-isop} and~\ref{large-bubble}.

\begin{proof}[Proof of Lemma \ref{alpha-ineq}]

Let us compare  $\sigma$ and $\tau$ with~$a$, for large enclosed volume. Write $u_1$ for $\,\max u\,$ and
$u_0$ for $\,\min u$, so we have $u_1-u_0\leq C$. The inequalities:
\[ u_1-C \leq u_0 \leq u_1 \leq u_0+ C\]
give rise to the inequalities:
\[
\left( 1-\frac{C}{u_1}\right)^n \leq \left(\frac{u_0}{u_1}\right)^n
\leq\frac{\sigma}{a}\leq\left(\frac{u_1}{u_0}\right)^n \leq \left(
1+\frac{C}{u_0}\right)^n
\; .\]
For large enclosed volume, $u_0$  and  $u_1$  are arbitrarily
 large compared to $C$, hence we may
assume $(1/2)a\leq\sigma (x)\leq 2a$ for all $x\in X$.

The function $u\mapsto u^n$ has monotone increasing derivative, therefore:
\[
\frac{u_1^n-u_0^n}{n^n}\;\leq\; \frac{n\, (u_1-u_0)\, u_1^{n-1}}{n^n}
\; \leq \; \frac{C\, u_1^{n-1}}{n^{n-1}}\; =\; C\,\left[\left(\frac{u_1}{n}\right)^n\right]^\alpha
\; \leq \; C\, (2a)^\alpha\; <\; 2\, C\, a^\alpha
\; ,\]
The interval $\big[\,(u_1/n)^n   \, ,\, (u_0/n)^n  \,\big]$ has length bounded by $2\, C\, a^\alpha$. Since it contains the number $a$ and all values of $\sigma$, we deduce:
\[ |\tau |\; =\; |\sigma -a|\; \leq\; 2\,C\, a^\alpha\; .\]

We estimate:
\begin{eqnarray*}
\frac{\sigma^{2\alpha -1}\,\tau+|\nabla\tau |^2}
{\sqrt{\sigma^{2\alpha}+|\nabla\tau |^2}} &=&
\sqrt{\sigma^{2\alpha}+|\nabla\tau |^2}-
\frac{a\,\sigma^{2\alpha -1}}
{\sqrt{\sigma^{2\alpha}+|\nabla\tau |^2}}\;\geq \\
&\geq& \sqrt{\sigma^{2\alpha}+|\nabla\tau |^2}
-a\,\sigma^{\alpha -1}
\; ,\end{eqnarray*}
and decompose the square root as
$\sigma^\alpha$ plus a multiple of $|\nabla\tau |^2\,$:
\[
\sqrt{\sigma^{2\alpha}+|\nabla\tau |^2}-\sigma^\alpha \; =\;
\frac{|\nabla\tau |^2}{\sqrt{\sigma^{2\alpha}+|\nabla\tau |^2}+\sigma^\alpha}
\; .\]
Let $H_1$ be an upper bound for mean curvature provided by the estimate (\ref{H-osc-bound})
in Theorem~\ref{teo-estimates}. Now  Theorem \ref{grad-estimate} in the Appendix provides a constant $C'$,
depending only on $n,X,H_1,C$, such that $|\nabla u|\leq C'$. Thus:
\[
\frac{|\nabla\tau |^2}{\sqrt{\sigma^{2\alpha}+|\nabla\tau |^2}+\sigma^\alpha}
\;\geq  \frac{|\nabla\tau |^2}{(\sqrt{1+C'^2}+1)\,\sigma^\alpha}
\;\geq\; C''\, a^{-\alpha}\, |\nabla\tau |^2
\; , \]
where $C''$ is a positive constant that depends only on $n,X,C$. We now have:
\[
\frac{\sigma^{2\alpha -1}\,\tau+|\nabla\tau |^2}
{\sqrt{\sigma^{2\alpha}+|\nabla\tau |^2}}
\;\geq\; C''\, a^{-\alpha}\, |\nabla\tau |^2
+\sigma^\alpha -a\,\sigma^{\alpha -1} \;=\;
C''\, a^{-\alpha}\, |\nabla\tau |^2
+\,\sigma^{\alpha -1}\,\tau
\; .\]
We further analyze:
\[
\sigma^{\alpha -1}\,\tau =(a+\tau )^{\alpha -1}\,\tau =
a^{\alpha -1}\left(1+\frac{\tau}{a}\right)^{-1/n}\,\tau
\; .\]
Since $|\tau |<2\, C\, a^\alpha$ and $\alpha <1$,  we may
assume  $-1/2<\frac{\tau}{a}<1/2$. But for $t\in (\, -1/2\,
,\, 1/2\, )$  it is $\left|\frac{d}{dt}(1+t)^{-1/n}\right|
<\frac{4}{n}$. It follows that:
\begin{equation}\label{4}
1-\frac{4}{n}\,\left|\frac{\tau}{a}\right| \;\leq\;
\left(1+\frac{\tau}{a}\right)^{-1/n} \;\leq\;
1+\frac{4}{n}\,\left|\frac{\tau}{a}\right|
\; .\end{equation}
Where $\tau\geq 0$, use the first inequality in (\ref{4}) and get:
\[
\left(1+\frac{\tau}{a}\right)^{-1/n}\, \tau \;\geq \;
\tau -\frac{4}{n}\,\frac{\tau}{a}\, \tau \; =\; \tau
-\frac{4}{n}\,\frac{\tau^2}{a}
\; .\]
where $\tau <0$, use the second inequality in (\ref{4}) and get:
\[
\left(1+\frac{\tau}{a}\right)^{-1/n}\, \tau \;\geq \;
\tau +\frac{4}{n}\,\left|\frac{\tau}{a}\right|\, \tau \; =\; \tau
-\frac{4}{n}\,\frac{\tau^2}{a}
\; .\]
So we have
$\sigma^{\alpha -1}\tau\geq a^{\alpha -1}\big(\tau -\frac{4}{n}\frac{\tau^2}{a}\big)$
everywhere, and we arrive at the inequality:
\[
\int_X
\frac{\sigma^{2\alpha -1}\,\tau+|\nabla\tau |^2}
{\sqrt{\sigma^{2\alpha}+|\nabla\tau |^2}}
\; \geq\; \int_X\left(\, C''\, a^{-\alpha}\, |\nabla\tau |^2
+a^{\alpha -1}\, \tau-\frac{4}{n}\, a^{\alpha -2}\,\tau^2\,\right)
\; ,\]
which yields Lemma~\ref{alpha-ineq} by using $\int_X\tau =0$ and taking the
factor $a^{-\alpha}$ out of the integral.
\end{proof}


\section{A special soap bubble family}\label{sec-estables}

In this section we prove the existence part of Theorem \ref{teo-estables}. The conditions that
make the construction of the family $\{ S_v\}$ possible are stated in detail
in Theorem~\ref{teo-family} below. The
ambient manifold is $M=X\times{\mathbb R}$, with $X$ a
suitable 2--dimensional Riemannian manifold.  In particular $M$ is
3-dimensional and our family consists of surfaces. They lie in a region where the Ricci curvature
is somewhere negative. It must be stressed that, in some of these families, the large soap 
bubbles are stable (proved in  Section~\ref{sec-stability})
but not  isoperimetric, i.e.\ the same amount of volume can be enclosed 
using less area. We shall also see that,
as the enclosed volume tends to infinity, their mean curvatures descend to a positive constant, not to zero.

There is an annulus $Y\subset X$ such that all the surfaces in the
family will be contained inside $Y\times{\mathbb R}$  (this
already prevents those surfaces from being of the form $X\times
S^0$),  and so we need only worry about the geometry
of the domain $Y\times{\mathbb R}$. Then $X$ can be any
closed Riemannian surface 
containing an isometric copy of~$Y$.

We describe $Y$ as $I\times S^1$,
where $I$  is an interval symmetric about 0.  Denote by  $s$
the coordinate along  $I$  and by
$\theta$  the angle coordinate along  $S^1\, =\,
[0,2\pi ]/_{0\sim 2\pi}$. Finally let  $y$  be the
coordinate along the ${\mathbb R}$ factor.

We endow $Y$  with a rotationally symmetric metric:
\begin{equation}\label{metric2}
G_c\; =\; ds^2+c^2\cdot f(s)^2\, d\theta^2
\; ,\end{equation}
Where $f(s)$ is positive and even,  that is $f(-s)=f(s)$,
and $c$ is a positive constant. The metric on
$Y\times{\mathbb R}$  is:
\begin{equation}\label{metric3}
G'_c\; =\; ds^2+c^2\cdot f(s)^2\, d\theta^2+dy^2
\; .\end{equation}
Notice that $Y\times{\mathbb R}$  has an isometric circle action,
defined by translating  $\theta$ by constants, and the following reflectional
symmetries:
\begin{equation}\label{the-symmetries}
(\, s\, ,\,\theta\, ,\, y\, )\;\longmapsto\; (\, -s\, ,\, \theta\, ,\, y\, )
\quad ,\quad
(\, s\, ,\,\theta\, ,\, y\, )\;\longmapsto\;  (\, s\, ,\,\theta\, ,\, \mbox{\rm const}-y\, )
\; .\end{equation}
The following auxiliary functions turn out to be very useful:
\[
F(s)=\int_0^sf(s)\, ds \quad ,\quad \varphi (s)=\frac{F(s)}{f(s)}
\; .\]
Since $f$ is even, both $F$ and $\varphi$ are odd. The basic identity
relating $f$ to $\varphi$ is:
\begin{equation}\label{f-phi-eq}
\frac{1-\varphi_s}{\varphi}\; =\; \frac{f_s}{f}
\; .\end{equation}
A function $\varphi (s)$ coming from this construction is not arbitrary:
it has to be an odd function,
vanish only at $\, s=0\,$,
 and satisfy \ $\varphi'(0)=1\,$.  Conversely any  $\varphi (s)$
meeting these three criteria comes from a positive even function; in fact,
the functions that $\varphi$ comes from are the members 
of the following one-parameter family:
\begin{equation}\label{f-phi-sol}
f(s)\; =\;\frac{c}{\varphi (s)}\cdot\exp\int\frac{ds}{\varphi (s)}
\quad ,\quad c\in {\mathbb R}^+
\; ,\end{equation}
The value  $s=0$  is the only
one where this formula may
pose a problem.  But if  $\varphi'(0)=1$  then  $\varphi
(s)\equiv s+s^2\,\tilde{\varphi} (s)$  for some smooth odd function
$\tilde{\varphi} (s)\,$,  and thanks to
$\frac{1}{s+s^2\tilde{\varphi} }=\frac{1}{s}-\frac{\tilde{\varphi} }{1+s\tilde{\varphi} }$  we
rewrite (\ref{f-phi-sol}) as:
\[
f(s)\; =\; \frac{c}{1+s\,\tilde{\varphi}  (s)}\cdot\exp\int_0^s\frac{-\tilde{\varphi} (s)\,
ds}{1+s\,\tilde{\varphi}  (s)}
\; ,\]
which is non-singular at $s=0$  and
defines a positive even function.

\begin{thm}\label{teo-family}
Suppose $\log F$ has a first inflection point at a value $s_0>0$. More concretely,
suppose $(\log F)_{ss}$ is negative in $0<s< s_0$ and positive in an interval starting at $s_0$.
Then $Y\times{\mathbb R}$ contains a family $\{ S_{s_1}\}_{0<s_1<s_0}$ of soap bubbles,
all contained in $(-s_0,s_0)\times S^1\times{\mathbb R}$,
and with enclosed volume going to infinity
as $s\to s_0$.
\end{thm}

Identity (\ref{f-phi-eq}) gives $(\log F)_{ss}= - (f^2/F^2)\,\varphi_s$. The hypothesis
in Theorem~\ref{teo-family} is equivalent to $\varphi_s$ being positive
in $[0,s_0)$ and negative in some interval starting at $s_0$. In particular
$\varphi$ has a first local maximum at $s=s_0$, see Figure~\ref{fig-funciones}.

\begin{figure}[ht]
\includegraphics[scale=0.55]{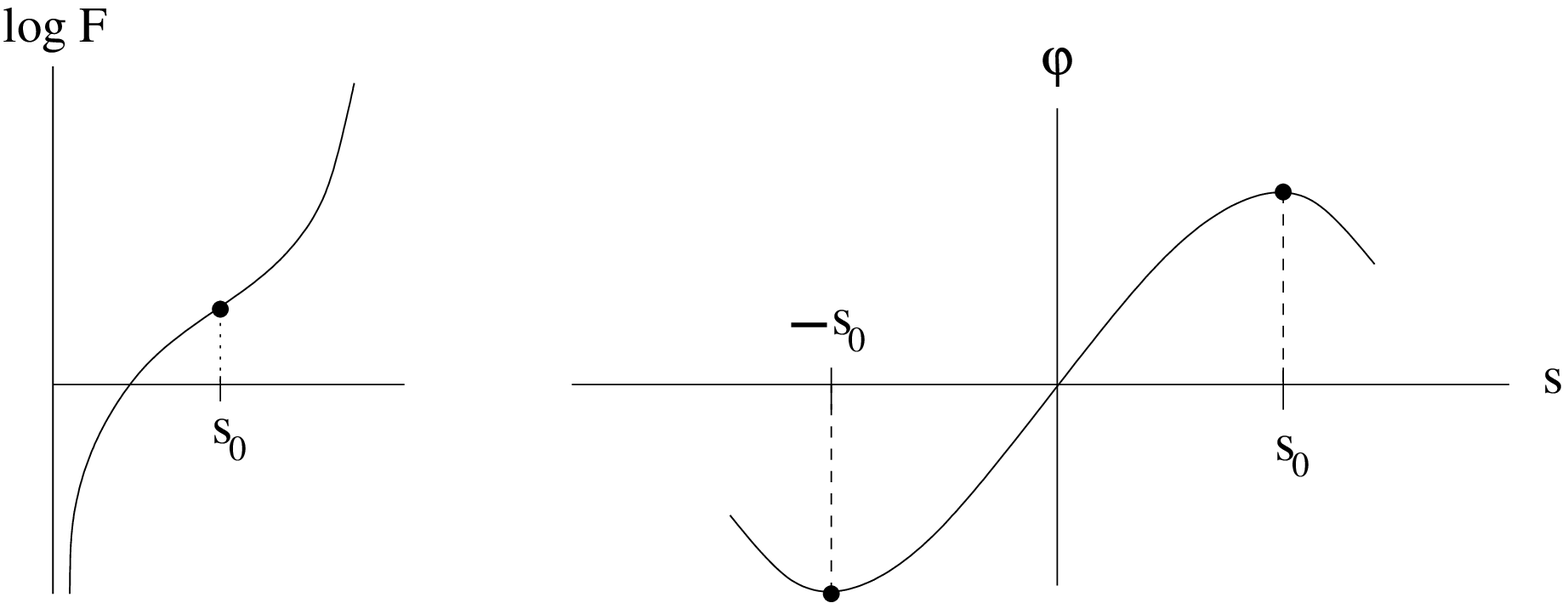}
\caption{}
\label{fig-funciones}
\end{figure}

Here is an example with $I=\big( -\pi +\varepsilon\, ,\, \pi -\varepsilon\big)$ and $s_0=\pi /2\,$:
\begin{equation}\label{ex-1}
F=2\,\tan \frac{s}{2}\;\; ,\;\; f=\frac{1}{\cos^2 (s/2)}\;\; ,\;\; \varphi
=\sin s
\; , \end{equation}
and another one with $I=\big( -1+\varepsilon\, ,\, 1-\varepsilon\big)$ and  $s_0=1/\sqrt{3}$:
\begin{equation}\label{ex-2}
F(s)=s\, (1- s^2)^{-1/2}\;\; ,\;\;
f(s)=(1- s^2)^{-3/2}\;\; ,\;\; \varphi (s)=s- s^3
\; .\end{equation}

By Theorem \ref{large-bubble}, we expect the surface piece $Y$ to have negative
Gaussian curvature somewhere in $(-s_0,s_0)\times S^1$. Let us directly
check this. Identity~(\ref{f-phi-eq}) implies:
\[ -\varphi\, f_{ss}\; =\; (f\,\varphi_s)_s-\varphi\,\frac{f_s^2}{f}\; ,\]
and for $s>0$ we get $ -\varphi\, f_{ss}\leq  (f\,\varphi_s)_s$.
Since the function $f\varphi_s$ is positive at $s=0$ and zero at $s=s_0$, its derivative has
to be negative in some interval $I_-\subset (0,s_0]$. Then $f_{ss}$ is positive in that same
interval. The Gaussian curvature of $Y$ equals $-f_{ss}/f$ and is thus negative in $I_0\times S^1$
and also in the image of this set under the reflection $(s,\theta )\mapsto (-s,\theta )$.

\begin{proof}[Proof of Theorem \ref{teo-family}]

We consider constant mean curvature surfaces
which are invariant under the circle action and the reflection
$(s,\theta ,y)\mapsto (s,\theta ,-y)$,
i.e.\ surfaces of the form:
\[
S_u \;=\; \{\;  \big( s,\theta ,\pm u(s)\big)\; :\; -s_1\leq s\leq s_1
\; ,\; 0\le\theta\leq 2\pi\;\}\; \approx\; C_u\times S^1
\; , \]
which is the result of rotating a profile curve $C_u=\{ y=\pm u(s)\}\subset
(sy\mbox{\rm \ plane})$. We also allow translates of these in the $y$-direction.  In
our construction  $u(s)$  will be an even
function defined in some symmetric interval $[-s_1,s_1]\subset I$ and satisfying:
\[ u(-s_1)=u(s_1)=0\quad ,\quad u'(-s_1)=+\infty \quad ,\quad u'(s_1)=-\infty\; ,\]
so that the closed profile $C_u$ is ${\mathcal C}^1$. The
condition for $S_u$ to have constant mean curvature is an ODE on the
profile curve $C_u$, and elliptic regularity implies that $C_u$ is
actually smooth. Then $S_u$ is a smooth surface
diffeomorphic with ${\mathbb T}^2$. Both $C_u$ and $S_u$ are embedded
if and only if $u(s)$ satisfies:
\[
u(s)>0\;\; ,\;\;\;\mbox{\rm for all }\; s\in (-s_1,s_1)
\; ,\]
in which case the closed profile bounds a disk region $D$ in the $sy$ plane. Then
$S_u$ is the boundary of the solid torus
$\{\, (s,\theta ,y)\, :\, (s,y)\in D\,\}\subset Y\times{\mathbb R}$
and is thus a soap bubble.

We choose for $S_u$ the
unit normal $\nu$ which points away from the solid torus.  The
mean curvature, as defined in Section~\ref{sec-monotonicity},
is given on  $S_u\cap\{ y>0\}$  by:
\begin{equation}\label{a}
H\; =\; \frac{-u''(s)}{\big( 1+u'(s)^2\big)^{3/2}}
\; +\;\frac{-u'(s)}{\big(
1+u'(s)^2\big)^{1/2}}\;\frac{f'(s)}{f(s)}
\; .\end{equation}
Notice that the equation is the same if we replace $f$ with
any positive constant multiple~$c\, f$. Once $\varphi$ is fixed, the
family $\{ c\cdot f\}_{c>0}$ is fixed and a solution to (\ref{a}) with $H$
constant defines a surface that has constant mean curvature $H$ with respect to all
the metrics~$G'_c$.

We can consider $\nu$ as lying flat on the $sy$-plane and
orthogonal to the profile.  We define the angle
$\alpha (s)$,  from the
$y$-axis  to  $\nu$,  as follows:
\[
-\frac{\pi}{2}\;\leq\,\alpha (s)\;\leq\;\frac{\pi}{2}
\quad ,\quad
\nu|_{y\geq 0}\; =\; \sin\alpha\,\partial_s+\cos\alpha\,\partial_y
\, ,\]
so that $\tan\alpha (s) = - u'(s)$ and (\ref{a}) becomes
$H\; =\; \frac{d}{ds}\,\big(\,\sin\alpha (s)\,\big)\, +\, (f_s/f)\sin\alpha
(s) $, equivalent to:
\begin{equation}\label{b}
\frac{d}{ds}\,\big(\, f(s)\,\sin\alpha (s)\,\big)\; =\; H\, f(s)
\; . \end{equation}
Since $f$ is even, a solution $u(s)$ to (\ref{a})  is even
if and only if $u'(0)=0$.   This
condition is equivalent to  $\alpha (0)=0$. The
conditions $u'(\pm s_1)=\mp\infty$ are equivalent to
$\,\sin\alpha (\pm s_1)=\pm 1$.

For a  constant value  $H$,  solutions to
(\ref{b}) with $\alpha (0)=0$ are given by:
\begin{equation}\label{sin-alpha}
\sin\alpha (s)=H\,\varphi (s)
\; .\end{equation}
The conditions $\sin\alpha (\pm s_1)=\pm 1$ are now equivalent to $H=1/\varphi (s_1)$. Since
we take $s_1\in (0,s_0]$, the constant $H$ takes values  in the interval 
$[\, 1/\varphi (s_0)\, ,\, +\infty )$; values smaller than $1/\varphi (s_0)$ will not 
appear in our construction. 

Expressing  $u'(s)=-\tan\alpha\,$  in terms of
$\,\sin\alpha$,  we arrive at:
\begin{equation}\label{c}
u(s)\;  =\; \mbox{\rm const} -\int_0^s\frac{\varphi (s)\, ds}
{\sqrt{\frac{1}{H^2}-\varphi (s)^2}}
\; =\; \mbox{\rm const} -\int_0^s\frac{\varphi (s)\, ds}
{\sqrt{\varphi (s_1)^2-\varphi (s)^2}}
\; ,\end{equation}
 an explicit formula that involves only $\varphi$. For each $s_1\in (0,s_0]$, the function:
\begin{equation}\label{int-phi}
u_{s_1}:(-s_1,s_1)\;\longrightarrow\;{\mathbb R}\quad ,\quad u_{s_1}(s)\; :=\;
 -\int_0^s\frac{\varphi (s)\, ds}
{\sqrt{\varphi (s_1)^2-\varphi (s)^2}}
\end{equation}
is the solution to (\ref{a}) with the following  data:
\[ H=1/\varphi (s_1)\quad ,\quad u(0)\; =\; 0\quad ,\quad u'(0)\; =\; 0\; .\]
We study its behavior in two cases: $s_1<s_0$ or $s_1=s_0$.

{\bf Case} $s_1<s_0$. For $s$ close to $s_1$, the integrand in (\ref{int-phi}) behaves like a positive multiple
of $(s_1-s)^{-1/2}$ because $\varphi(s)^2$ has positive derivative at $s=s_1$. Since
$\,\int_{s_1-\varepsilon}^{s_1} (s_1-s)^{-1/2}\, ds$ is finite,
$u_{s_1}(s)$ has a finite limit as $s\to s_1$. It has the same finite limit as $s\to -s_1$, and so it
extends to the closed interval $[-s_1,s_1]$. The extended function is negative except at $s=0$,
and achieves its minimum at the endpoints $\pm s_1$. Then the function
\[  \tilde{u}_{s_1}:[-s_1,s_1]\;\longrightarrow\;{\mathbb R}
\quad ,\quad \tilde{u}_{s_1}(s)\; :=\;  u_{s_1}(s)-u_{s_1}(\pm s_1)
\]
is the solution to (\ref{a}) with data:
\[ H=1/\varphi (s_1)\quad ,\quad \tilde{u}(-s_1)\; =\; \tilde{u}(s_1)\; =\; 0\; , \]
and satisfies $\tilde{u}_{s_1}(s)>0$ for $s\in (-s_1,s_1)$. The graph of $\tilde{u}_{s_1}$
meets the graph of $- \tilde{u}_{s_1}$ only at the endpoints $(\pm s_1,0)$, where
the derivative is infinite. By the previous discussion,
the surface $S_{s_1}:=S_{\tilde{u}_{s_1}}$ is a soap bubble in $Y\times{\mathbb R}$ with respect
to all metrics~$G'_c$.

Doing this for all $s_1\in (0,s_0)$, we get a soap bubble family $\{ S_{s_1}\}_{0<s_1<s_0}$. Each
member $S_{s_1}$ of this family  is contained
in the part $[-s_1,s_1]\times S^1\times{\mathbb R}$, hence
they all lie inside $(-s_0,s_0)\times S^1\times{\mathbb R}$ which is a proper subset of $Y\times{\mathbb R}$.

{\bf Case} $s_1=s_0$. Since $\varphi (s)^2$ has a local maximum at $s=s_0$, for $s$ close
to $s_0$ the integrand in (\ref{int-phi}) is at least as large as a
positive multiple of  $(s_0-s)^{-1}$. From
$\int_{s_0-\varepsilon}^{s_0} (s_0-s)^{-1}\, ds=+\infty$ we then deduce
that the function $u_{s_0}$ tends to $-\infty$, at least at a logarithmic rate, as $s\to s_0$
and also as $s\to -s_0$. The undergraph:
\[  E\; =\; \{\; (s,y)\; :\: y\leq u_{s_0}(s)\;\} \;\subset\; (-s_0,s_0)\times{\mathbb R}\; , \]
has infinite area  both in the standard area measure and in
the measure $c\, f(s)\, dsdy$. 

For $s_1<s_0$, define a closed profile $\widetilde{C}_{s_1}$ as the union of the graphs of $u_{s_1}$ and of
the reflected function $2\, u_{s_1}(\pm s_1)-u_{s_1}(s)$. The point $(0,0)$ is where $y$ is maximum
on each $\widetilde{C}_{s_1}$.  Denote by $\Omega_{s_1}$
the region bounded by $\widetilde{C}_{s_1}$ in the $sy$ plane.

For fixed $s$ the integral (\ref{int-phi}) is an increasing function
of $s_1$. This implies that the D-shaped domains  
\[ D_{s_1}\; =\ \{\, (s,y)\, :\, \min u_{s_1}\leq y\leq u_{s_1}(s)\, ,\, -s_1\leq s\leq s_1\,\} \;  . \] 
expand as $s_1\nearrow s_0$,
and they fill up $E$. A fortiori, the O-shaped regions $\Omega_{s_1}$ also fill up $E$, as shown in Figure~\ref{fig-unbounded}. Therefore the area of $\Om_{s_1}$  in the measure $c\, f(s)\, dsdy$   goes to infinity as $s_1\to s_0$, and so does
the volume enclosed by $S_{s_1}$  in~$Y\times{\mathbb R}$.

The mean curvature $H=1/\varphi (s_1)$ decreases as $s_1\to s_0$, but the limit is the positive numbert~$1/\varphi (s_0)$.
Recall that values $H<1/\varphi (s_0)$ never appear in this construction.

\begin{figure}[ht]
\includegraphics[scale=0.5]{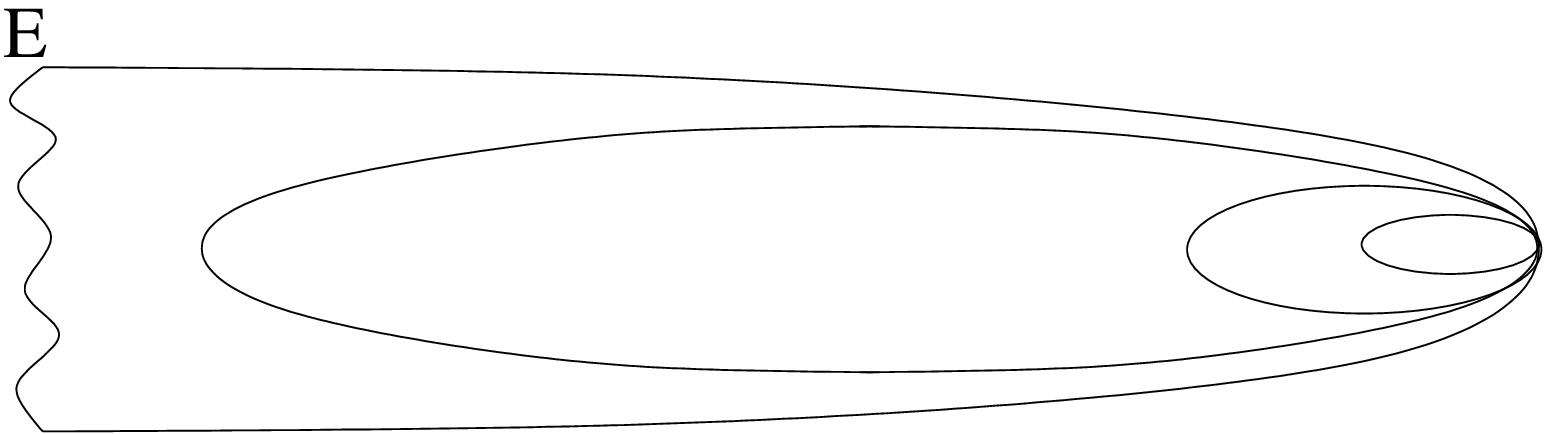}
\caption{}
\label{fig-unbounded}
\end{figure}

We have a diffeomorphism $(0,s_0)\to (0,+\infty)$ that maps $s_1\in (0,s_0)$ to the volume
enclosed by $S_{s_1}$. So we can use the enclosed volume as parameter in place of $s_1$. In
this way the soap bubble family becomes $\{ S_v\}_{v>0}$, as stated in
Theorem~\ref{teo-estables}.

\end{proof}


\section{Stability in the special family}\label{sec-stability}

In this section we prove the stability part of Theorem~\ref{teo-estables}. The following
theorem explicitly gives a condition under which
the large surfaces in the family $\{ S_v\}$ are stable. We shall also show that such
condition is satisfied in some examples.

\begin{thm}\label{teo-estabilidad}
Let $f,F,\varphi ,s_0$ be as in Theorem \ref{teo-family}. If $(\log f)_{ss}>0$ in $(-s_0,s_0)$,
then there is a number $\beta >0$ such that if
$\varphi (s_1)/c>\beta$ then the soap bubble $S_{s_1}$ is
stable with respect to the metric $G'_c$.
\end{thm}

Fix a value $c<\varphi (s_0)/\beta$ and let $s_2\in (0,s_0)$ be the solution to $\varphi (s_2)/c =\beta$.
Then the soap bubbles $S_{s_1}$ with $s_1>s_2$ are all stable with respect to $G'_c$. The small
$S_{s_1}$, corresponding to $s_1$ close to zero, look like thin tubes
around a circle and are not stable.

In example (\ref{ex-1}) we have $(\log f)_{ss}=f/2>0$.  In example
(\ref{ex-2}) we have
$(\log f)_{ss}=(3+3\, s^2)/(1-s^2)^2>0$. Both examples provide stable soap bubbles
enclosing arbitrarily large volume in $Y\times{\mathbb R}$ but whose projection to $Y$ is not surjective.

We do not know whether for $c\geq\varphi (s_0)/\beta$ the metric $G'_c$ admits a family of non-isoperimetric
stable soap bubbles with enclosed volume going to infinity.

Consider for a moment the problem of embedding $(Y,G_c)=(I\times S^1,G_c)$ isometrically into ${\mathbb R}^3$
as a surface of revolution. Such embedding would be:
\[ (x_1,x_2,x_3)\; =\;\big(\; r(s)\cos\theta\; ,\; r(s)\sin\theta\; ,\; x_3(s)\;\big)\; ,\]
for some functions $r(s),x_3(s):I\to{\mathbb R}$ satisfying the following equations:
\[ r'(s)^2+x_3'(s)^2\; =\; 1\quad ,\quad r(s)\; =\; c\cdot f(s)\; .\]
The function $r(s)$ is already given by the second equation, while $x_3(s)$ is given
by  the formula $x_3(s)=\int\sqrt{1-c^2\, f'(s)^2}\, ds$. Therefore $x_3(s)$
exists if $c\, |f'(s)|<1$, which is true for small enough $c$. Once $Y=I\times S^1$ is
thus embedded into ${\mathbb R}^3$, we can extend it to a closed surface $X\subset{\mathbb R}^3$.
If $c$ is also smaller than $\varphi (s_0)/\beta$, then $X\times{\mathbb R}$ is a cylinder in ${\mathbb R}^4$
admitting the stable family of soap bubbles.

\begin{proof}[Proof of Theorem \ref{teo-estabilidad}]

Let us first obtain a convenient formula for the index form of $S_u$. We have
the orthogonal bases
$\{\,\partial_s+u'(s)\,\partial_y\, ,\,\partial_\theta\,\}$ for
the tangent spaces of  $S_u$,  and we consider orthonormal
bases  $\{{\bf e}_1,{\bf e}_2\}$  with  ${\bf e}_1\in{\mathbb
R}(\partial_s+u'(s)\,\partial_y)$  and  ${\bf e}_2\in{\mathbb
R}\partial_\theta$.  In these orthonormal bases, the matrix of
$II$ is:
\[
\left[\begin{array}{cc}\frac{u_{ss}}{(1+u_s^2)^{3/2}} & 0
\\ 0 & \frac{f_s}{f}\cdot\frac{u_s}{(1+u_s^2)^{1/2}}\end{array}\right]
\; =\;
\left[\begin{array}{cc} - \frac{d}{ds}\,\sin\alpha  & 0\\
0 & -\frac{f_s}{f}\cdot\sin\alpha\end{array}\right]
\; ,\]
independent of the constant $c$. Using
equalities (\ref{f-phi-eq}) and (\ref{sin-alpha}),  we rewrite that matrix as:
\[
\left[\begin{array}{cc}-H\varphi_s & 0
\\ 0 & H\cdot (\varphi_s-1)\end{array}\right]
\; ,\]
and so
$|II|^2=H^2\,\big(\,{\varphi_s}^2+(\varphi_s-1)^2\,\big)$.

For the Ricci term we have  $\displaystyle\mbox{\rm Ric}\, (\nu ,\nu
)=\frac{-f_{ss}}{f}\,\sin^2\alpha =\frac{-f_{ss}}{f}\,
H^2\,\varphi^2$, also independent of $c$. The relations
$\frac{f_s}{f}=\frac{1-\varphi_s}{\varphi}$ \ and \ $\frac{f_{ss}}{f}=
\frac{-\varphi_{ss}}{\varphi}+\frac{(1-\varphi_s)(1-2\varphi_s)}{\varphi^2}$
lead to:
\[
-\mbox{\rm Ric}\, (\nu ,\nu )\; =\; H^2\cdot\big(\,
-\varphi\,\varphi_{ss}+(1-\varphi_s)\, (1-2\varphi_s) \,\big)
\; .\]
Then, after a trivial simplification:
\[
P\; =\; -\mbox{\rm Ric}\, (\nu ,\nu )-|II|^2\; =\;
H^2\cdot (-\varphi\,\varphi_{ss}-\varphi_s)
\; .\]
We can use $(s,\theta)$  as coordinates in  $S_u\cap\{ y\geq 0\}$  and in
$S_u\cap\{ y\leq 0\}$; accordingly we view $ds\, d\theta$ as a measure on all of $S_u$.  In these
coordinates:
\begin{eqnarray*}
\mbox{\rm metric induced on }S_u &=& (1+{u'}^2)\, ds^2+c^2\, f^2\, d\theta^2\; , \\
\mbox{\rm area measure on }S_u  &=& \sqrt{1+{u'}^2}\; c\,  f\; ds\, d\theta
\; , \end{eqnarray*}
and for any $g:S_u\to{\mathbb R}$ we find that:
\[
|\nabla^{S} g|^2\, d\,\mbox{\rm area} \; =\; \left(\frac{c\, f}{
\sqrt{1+{u'}^2} } \, g_s^2+\frac{\sqrt{1+{u'}^2}}{c\, f}\,
g_\theta^2\right)\, ds\, d\theta
\; .\]
For $S_{s_1}$ we can write
$1+u'^2=1+\frac{\varphi^2}{\varphi
(s_1)^2-\varphi^2}=\frac{\varphi (s_1)^2}{\varphi
(s_1)^2-\varphi^2}$, and finally obtain:
\begin{equation}\label{Q1}
Q(g)\; =\; \int_S\left(\; A\, g_s^2+\frac{1}{A}\, g_\theta^2+B\, g^2
\right)\, ds\, d\theta
\; , \end{equation}
where:
\begin{equation}\label{Q2}
A\; =\; H\, c\, f(s)\,\sqrt{\varphi (s_1)^2-\varphi (s)^2}\quad ,\quad
B\; =\; H\, c\, f(s)\,\frac{-\varphi\,\varphi_{ss}-\varphi_s}{\sqrt{\varphi (s_1)^2-\varphi (s)^2}}
\; .\end{equation}
In $S_{s_1}\cap\{ y\neq 0\}$, where $s$ is a valid coordinate, the
Jacobi operator $\Delta^S g-P\, g$ is the result of multiplying the following operator with a positive function:
\begin{equation}\label{Jacobi-coords}
\left( A\, g_s\right)_s+ \left(\frac{g_\theta}{ A }\right)_\theta -B\, g
\; . \end{equation}
The surface $S_{s_1}$ is invariant under the reflections:
\[ (s,\theta ,y)\;\longleftrightarrow\; (-s,\theta ,y)\quad ,\quad
(s,\theta ,y) \;\longleftrightarrow\; (s,\theta ,-y)
\; ,\]
and so it makes sense to define, for functions $g:S_{s_1}\to{\mathbb R}$, the properties of
being odd or even in the $s$ variable, and the same for the $y$ variable. We shall denote by
$g_{[s]},g_{(s)}$ the odd and even parts of $g$ with respect to $s$, respectively, that is
\[
g_{[s]} := \frac{1}{2}\, \big(\, g(s,\theta,y)-g(-s,\theta ,y)\,\big)
\quad ,\quad
g_{(s)} := \frac{1}{2}\, \big(\, g(s,\theta,y)+g(-s,\theta ,y)\,\big)
\; ,\]
and $g_{(y)},g_{[y]}$ shall have the analogous meaning in the $y$ variable. In particular,
we shall use the decomposition:
\[ g\; =\; g_{[y]}+g_{(y)[s]}+g_{(y)(s)}\; .\]
The polar bilinear form of the index form admits the expression:
\begin{equation}\label{polar}
Q(g,\tilde{g})\; =\; \int_S\left(\; A\, g_s\,\tilde{g}_s+\frac{1}{A}\, g_\theta\,\tilde{g}_\theta +B\, g\,\tilde{g}
\right)\, ds\, d\theta
\; ,\end{equation}
and it is obvious, by the symmetries of $A$ and $B$, that functions with different parity in $s$ or in $y$
are $Q$-orthogonal. Therefore, for every $g:S_{s_1}\to{\mathbb R}$ we have:
\[ Q(g)\; =\; Q\big(g_{[y]}\big) +Q\big(g_{(y)[s]}\big) +Q\big(g_{(y)(s)}\big) \; ,\]
and we shall do a separate study of the positivity of each summand. The functions
$g_{[y]}$ and $g_{(y)[s]}$ always have zero average, hence
$g$ has zero average if and only if $g_{(y)(s)}$ has zero average.

The next result follows from the identity $\gamma^2\, |\nabla g_1|^2=
|\nabla (g_1\gamma)|^2-\nabla\gamma\cdot\nabla (g_1^2\, \gamma )$.

\begin{lem}\label{G-F}
Let $\Sigma$ be any compact surface with a Riemann metric. For any two functions $g_1,\gamma:\Sigma\to{\mathbb R}$
the following holds:
\begin{eqnarray*}
&&\int_\Sigma\big(\,|\nabla^\Sigma (g_1\,\gamma )|^2+P\, (g_1\,\gamma )^2   \,\big) \; = \\
&=& \int_\Sigma g_1^2\,\gamma\cdot\big(\,P\gamma-\Delta^\Sigma\gamma\,\big) +
\int_\Sigma\gamma^2\, |\nabla^\Sigma g_1|^2
+\int_{\partial\Sigma} (g_1^2\,\gamma )\;\eta\cdot\nabla^\Sigma\gamma
\; ,\end{eqnarray*}
where $\eta$ is the outer conormal along $\partial\Sigma$.
\end{lem}

{\bf First case: $g$ is odd in the variable $y$}. Since $\partial_y$ is a Killing vector field, the
function $\psi = (1/H)\,\langle\partial_y,\nu\rangle$ is a solution to the Jacobi equation:
\[  \Delta^S\psi -P\,\psi =0 \; .\]
The formula $\psi =(\mbox{\rm sig } y)\cdot\sqrt{\varphi (s_1)^2-\varphi^2}$ shows that $\psi$ vanishes
with non-zero derivative along the two circles defined as $S_{s_1}\cap\{ y=0\}$. Any function $g$ that is odd in the
variable $y$ vanishes along those circles too, hence $g=\psi\cdot g_1$ for some
smooth function $g_1$  on $S_{s_1}$. In this case Lemma \ref{G-F} gives:
\[ Q (g)\; =\; Q\big(\psi\cdot g_1\big)\; =\; \int_S \psi^2\,|\nabla^Sg_1|^2\; .\]
It follows that, for $g$ odd in $y$, the number $Q(g)$ is positive unless $g$ is
a constant multiple of $\langle\partial_y,\nu\rangle$.

{\bf Second case: $g$ is even in $y$ and odd in $s$}. Now we have $g=\varphi\, g_1$ for some smooth function $g_1$
on $S_{s_1}$. Using formula (\ref{f-phi-eq}) one can do
a direct calculation that yields the following result:
\[
(A\varphi_s)_s\; =\; B\,\varphi- \frac{c\, f\, \varphi (s_1)}{\sqrt{\varphi (s_1)^2-\varphi^2}}
\cdot\varphi\cdot (\log f)_{ss}
\; .\]
The hypothesis of Theorem \ref{teo-estabilidad} then says that $(A\,\varphi_s)_s-B\,\varphi$
and $\Delta^S\varphi -P\,\varphi$ are negative multiples of $\varphi$. In this case Lemma~\ref{G-F}
gives us the following:
\[
Q(g)\; =\; Q(g_1\varphi )\; =\; \int_{S_{s_1}} g_1^2\,\varphi^2\cdot(\mbox{\rm positive})
+\int_{S_{s_1}} \varphi^2\, |\nabla^Sg_1|^2
\; ,\]
and it is obvious that $Q(g)>0$ unless $g_1$ and $g$ are identically zero.

{\bf Third case: $g$ is even in both $s$ and $y$}. There is a closed profile $C_u\subset (sy\;\,\mbox{\rm plane})$
such that $S_{s_1}$ is like $C_u\times S^1$ with the coordinate $\theta$ going along the $S^1$ factor. Moreover $C_u$
is the union of two graphs $\{ y=\pm u(s)\}$ with $-s_1\leq s\leq s_1$. Consider the Fourier expansion:
\[ g\; =\; a_0+\sum_{k\geq 1} \big(\, a_k\,\cos k\theta+b_k\,\sin k\theta\,\big) \; ,\]
where the coefficients are functions $a_k,b_k:C_u\to{\mathbb R}$ as symmetrical as $g$ is:
\[ a_k(-s,y)=a_k(s,y)\;\;\mbox{\rm and }\; a_k(s,-y)=a_k(s,y)\, ,\; \mbox{\rm same for }\; b_k\; .\]
Considering $ds$ as a measure on all of $C_u$, we can write:
\begin{eqnarray*}
Q(g) &=& \int_{C_u}2\pi\,\big[\,  A\, a_{0s}^2+B\, a_0^2  \,\big]\, ds + \\
&& + \int_{C_u}\pi\,\left[\, A\,\sum_{k\geq 1}(a_{ks}^2+b_{ks}^2)
+\sum_{k\geq 1}\left(\,\frac{k^2}{A}\, (a_k^2+b_k^2)+ B\, (a_k^2+b_k^2)
\,\right)\,\right]\, ds
\; , \end{eqnarray*}
thus:
\[
Q(g)\; \geq\; 2\pi\,\int_{C_u} (\,A\,a_{0s}^2+B\, a_0^2\, )\,ds +
\pi\,\sum_{k\geq 1}\int_{C_u} \left( \frac{1}{A}+B \right)\, (a_k^2+b_k^2)\, ds
\; .\]
The points
\[ (s_1,0)\; ,\; (0, u(0))\; ,\; (-s_1,0)\; ,\; (0,-u(0)) \]
separate $C_u$ into four quadrants. That $g$ has zero average in $S_{s_1}$ is equivalent to $a_0$ having
zero average in $C_u$ and, by the symmetries, it
also has zero average on each quadrant. The first quadrant is the graph
$C'_u=\left\{ \, \big( s,u(s)\big)\, :\, 0\leq s\leq s_1\,\right\}$ and, since $a_0$ has
zero average on it, there is a value $\overline{s}\in (0,s_1)$ such
that $a_0\big(\overline{s},u(\overline{s})\big)=0$.

We shall now use
the one-dimensional version of Lemma~\ref{G-F}, see \cite[page 107]{GF}. For
$0\leq s\leq \overline{s}$ use $\gamma=\psi$ and get:
\begin{eqnarray*}
&& \int_{C'_u\cap\{ 0\leq s\leq\overline{s}\}}(\,A\,a_{0s}^2+B\, a_0^2\, )\,ds \; =\;
\int_{C'_u\cap\{ 0\leq s\leq\overline{s}\}} A\,\psi^2\,\left[\frac{d}{ds}\,\frac{a_0}{\psi}\right]^2\, ds +
\left[\,A\,\frac{\psi_s}{\psi}\, a_0^2\,\right]_0^{\overline{s}}\;\geq \\
&& \geq \left[\,A\,\frac{\psi_s}{\psi}\, a_0^2\,\right]_0^{\overline{s}} \; =\; 0
\; , \end{eqnarray*}
and we see that $\int_{C'_u\cap\{ 0\leq s\leq\overline{s}\}}(\,A\,a_{0s}^2+B\, a_0^2\, )\,ds \geq 0$, with
strict inequality unless $a_0|_{0\leq s\leq\overline{s}}$ is a constant multiple of~$\psi$.  For 
$\overline{s}\leq s\leq s_1$ use
$\gamma=\varphi$ and obtain:
\[
\int_{C'_u\cap\{ \overline{s}\leq s\leq s_1\}}(\,A\,a_{0s}^2+B\, a_0^2\, )\,ds \;\geq\;
\left.-A\,\frac{\varphi_s}{\varphi}\, a_0^2\right|_{s=\overline{s}} \; =\; 0
\; ,\]
with strict inequality unless $a_0\equiv 0$ on $C'_u\cap\{\overline{s}\leq s\leq s_1\}$.

But if $a_0|_{C'_u}$ has to be a constant multiple of $\psi$ on $0\leq s\leq\overline{s}$, 
and zero on $\overline{s}\leq s\leq s_1$, then it must be identically zero We 
conclude that $\int_{C'_u}(\,A\,a_{0s}^2+B\, a_0^2\, )\,ds>0$ unless
$a_0$ is everywhere zero.

We want the term $\pi\,\int_C\left(\frac{1}{A}+B\right)\, (a_k^2+b_k^2)\, ds$ to be
positive unless $a_k=b_k=0$ for $k\geq 1$. Thus we want to ensure that $\frac{1}{A}+B$
is a positive function. The formulas:
\begin{eqnarray*}
\frac{1}{A} &=& \frac{1}{H\, c\, f\, \sqrt{\varphi (s_1)^2-\varphi^2}}\; =\;
\frac{\varphi (s_1)}{c}\,\frac{1}{f\,\sqrt{\varphi (s_1)^2-\varphi^2}}\; , \\
B &=& H\,c\,f\,\frac{-\varphi\,\varphi_{ss}-\varphi_s}{\sqrt{\varphi (s_1)^2-\varphi^2}} \; =\;
\frac{c}{\varphi (s_1)}\, \frac{1}{f\cdot\sqrt{\varphi (s_1)^2-\varphi^2}}\cdot  f^2\cdot
(-\varphi\,\varphi_{ss}-\varphi_s)
\; , \end{eqnarray*}
show that the following condition:
\[ \left(\frac{\varphi (s_1)}{c}\right)^2\; >\; \beta^2\; :=\;
\max_{[0,s_0]} \big(\, f^2\cdot|-\varphi\,\varphi_{ss}-\varphi_{s_1}|\,\big)
\; ,\]
implies that  $\frac{1}{A}+B$  is a positive function. This takes care of the third case.

Putting the three cases together we see that, for
$\varphi (s_1)/c>\beta$ and $g:S_{s_1}\to{\mathbb R}$
with zero average, it is $Q(g)>0$ unless $g$ is a constant
multiple of $\langle\partial_y,\nu\rangle$. This proves
Theorem~\ref{teo-estabilidad} and finishes the proof of Theorem~\ref{teo-estables}.

\end{proof}


\section*{Appendix: slope estimate}

We prove here a gradient estimate for a rotated graph $S=\{\rho =u(x)\}\subset X\times\R$
that has constant mean curvature $H$ and  is some distance apart from its symmetry axis $X\times\{ 0\}$. The gradient
bound depends on the radius oscillation of~$S$.

More concretely, the radius function $u:X\to{\mathbb R}$ is defined on all of $X$
and we assume $\,\min u>1$. The radius oscillation is the number:
\[ C\; =\; \max u -\min u\; .\]

\begin{thm}\label{grad-estimate}
There is a constant $C'$, depending only on $n,X,C$, such that:
\[  |\nabla u (x)|\leq C'\, ,\;\;\mbox{\rm for all }\,x\in X
\; .\]
\end{thm}

First we prove the theorem for $n\geq 2$. Points on $X\times\R$ will be described as $p=(x,\overline{y},y_n)$,
where $x\in X$, $\overline{y}=(y_1,\dots ,y_{n-1})\in{\mathbb R}^{n-1}$, and $y_n\in{\mathbb R}$. The function:
\[ f(x,\overline{y})\; =\; \sqrt{u(x)^2-|\overline{y}|^2}  \; ,\]
already considered in formula (\ref{f-graph}) of Section~\ref{sec-regularity}, is defined in an open subset
of $X\times{\mathbb R}^{n-1}$ that contains the closure of the following domain:
\[ U_1\; =\; \{\, (x,\overline{y})\; :\; |\overline y|< 1\,\}
\; =\; X\times B^{{\mathbb R}^{n-1}}(0,1)\; , \]
and $S_+:=S\cap\{ |\overline{y}|< 1\; ,\; y_n>0\}$ is described
as a Cartesian graph:
\[ S_+\; =\; \{ y_n=f(x,\overline{y})\, ,\, (x,\overline{y})\in U_1\}\; .\]
The function $f$ has the following bounds:
\[ \min u -1\; <\; f\;\leq\max u \; .\]
A gradient bound for $f$ will provide the same for $u$, because $u(x)=f(x,0)$.

Denote by $\nu$  the unit normal of $S$ which along $S_+$ points to the
positive $y_n$-direction. Consider also the function:
\[ w\; =\; \sqrt{1+|\nabla f|^2}\; .\]

We use the method of N. Korevaar in  \cite{Kor}. Let
$\lambda  :X\times{\mathbb R}^n\rightarrow{\mathbb R}$
be a non-negative continuous function, smooth at the points of $S$ where it is positive,
and zero  on $S\cap\{ |\overline{y}|\geq 1\}$.  For small $\varepsilon >0$
push $S$ off itself by taking
each point $z\in S$
to the endpoint $z_\varepsilon$  of the geodesic segment with length
$\varepsilon\,\lambda  (p)$ and initial velocity $\nu_p$.
This deforms $S$  to a new hypersurface $S_\varepsilon$
whose mean curvature we denote~$H_\varepsilon$.

The new hypersurface $S_\varepsilon$  lies above  $S$ in $\{ y_n>0\}$ and
coincides with $S$ outside  $\{ |\overline{y}|\geq 1+\mbox{\rm O}(\varepsilon )\}$.
The points of $S\cap\{ |\overline{y}|\geq1\}$ are not moved at all.
Therefore the height of  $S_\varepsilon$ over $S$  is
maximized at some point $z_{0,\varepsilon}\in S_\varepsilon$ which is the
end of a geodesic segment issuing from some
$z(\varepsilon )\in S_+$.

We can apply to  $S_\varepsilon$  a downward translation in the
$y_n$-direction until it touches $S$
from underneath at the translated point of $z_{0,\varepsilon}$.  Hence
$H_\varepsilon (z_{0,\varepsilon} )\geq H$.  But formula
(\ref{formula-H.}) of Section~\ref{sec-regularity} gives the
following expression for $H_\varepsilon$:
\[
H_\varepsilon (z_\varepsilon ) =
H\; -\;\left(\, | II |^2\,\lambda
+\mbox{\rm Ric}\, (\nu ,\nu )\,\lambda
+\Delta^{S}\lambda \,\right)_z\cdot\varepsilon
\; +\; {\mathcal E}_1\, ,\mbox{\rm \ for all}\; z\in S\cap\{ \lambda >0\}
\; .\]
The error term ${\mathcal E}_1$
is an $O(\varepsilon^2)$ depending
at most on third derivatives of $(f,\lambda ,G)$;
here $G$ denotes the metric on $X$.  The
inequality $H_\varepsilon (z_{0,\varepsilon} )\geq H$  then
implies:
\begin{equation}\label{ineq-laplacian}
\left(\,\Delta^{S}\lambda  -
R_0\,\lambda \,\right)_{z(\varepsilon )} \;\leq\;
\frac{1}{\varepsilon}\,{\mathcal E}_1
\; ,\end{equation}
where $R_0$ is the number defined
in formula (\ref{def-R0}) of Section~\ref{sec-estimates-2}, and
$\frac{1}{\varepsilon}\,{\mathcal E}_1$  is an
$O(\varepsilon )$  depending at most on third derivatives of
$(f,\lambda ,G)$; it may be positive or negative.

We shall see that (\ref{ineq-laplacian}) leads to an apriori bound
$w\big( z(\varepsilon )\big)\leq w_0$ at
the special point $z(\varepsilon )$. The idea is the following:
Korevaar constructs  $\lambda $  with large
positive second derivative  $\lambda _{y_n\, y_n}$; if  $S$
were close to vertical at $z(\varepsilon )$, then $\left(\Delta^{S}\lambda\right) \big( z(\varepsilon )\big)$
would be a large positive number and (\ref{ineq-laplacian}) would
be contradicted for small $\varepsilon$.

The number $w_0$
is independent of $\varepsilon$. Let us see that, in the limit
as $\varepsilon\to 0$, the quantity $\lambda  w$ is maximized at~$z(\varepsilon )$.

For $\varepsilon$ sufficiently small, depending on first
derivatives of $(f,\lambda ,G)$, the hypersurface
$S_\varepsilon\cap\{|\overline{y}|\leq 1\, ,\, y_n>0\}$ is the graph of a
function $f_\varepsilon (x,\overline{y})$. Define  $f(z)=f(x,\overline{y})$ if
$z=(x,\overline{y},y_n)$, and similarly for $f_\varepsilon$
and  $w$. We have the following formula for the height of
$S_\varepsilon$  above~$S$:
\[
f_\varepsilon (z) -f(z)\; =\; \varepsilon \,\lambda  (z)\, w(z)
\; +\; {\mathcal E}_2
\; ,\]
where the error term ${\mathcal E}_2$  is an
$O(\varepsilon^2)$  depending at most on second derivatives of
$(f,\lambda ,G)$.  See the picture in \cite[page 85]{Kor} for a
convincing proof. Since the quantity $(f_\varepsilon
-f)/\varepsilon$ is maximized at $z(\varepsilon )$,  we have:
\begin{equation}\label{ineq-height}
\lambda  (z)\, w(z)\;\leq\; \lambda  \big( z(\varepsilon )\big)\, w\big( z(\varepsilon )\big)
\; +\; {\mathcal E}_3
\quad ,\quad\mbox{\rm for all }\; z\in S_+
\; ,\end{equation}
where
${\mathcal E}_3$  is an  $O(\varepsilon )$  depending at most
on second derivatives of  $(f,\lambda ,G)$.  As we make
$\varepsilon\to 0$  the point  $z(\varepsilon )$ keeps moving inside
$S_+$ and the term ${\mathcal E}_3$
tends to zero. Given the bound $w\big( z(\varepsilon )\big)\leq w_0$, obtained
from (\ref{ineq-laplacian}) at the special points~$z(\varepsilon )$,  in the
limit we have:
\[
\lambda  (z)\, w(z)\leq \left( \max_{S_+}\,\lambda \right)\cdot w_0
\; ,\]
which gives a slope bound at those  $z\in S_+$ with  $\lambda  (z)$
not too small. Consequently, to the conditions already imposed on $\lambda $ we
add the following one: \
$\lambda  >0$  on $S_+\cap\{\overline{y}=0\}$.

The function $\lambda $  is first given by the ansatz
$\lambda (z)\equiv e^{C_1\,\mu (z) }-1$,  so that  $\lambda $ is positive
where $\mu $  is positive and is zero where  $\mu $  is zero.  We
compute:
\[
\Delta^{S}\lambda  \; =\; e^{C_1\,\mu }\cdot C_1\cdot
\left(\, \Delta^{S}\mu +C_1\,|\nabla^{S}\mu  |^2\,\right)
\; .\]
The factor $e^{C_1\,\mu }\, C_1\geq C_1$
is going to be large.  We want
the expression in parenthesis to be large at steep points
of~$S_+$.

The function $\mu $  is given by a second ansatz:
\[
\mu  (x,y)\;\equiv\;\left(\,
1-|\overline{y}|^2-\frac{\big( y_n-\min f\big)^+}{2+2C}\,\right)^+
\; ,\]
notice that it does not depend on the point  $x\in X$.   Both  $\mu $  and  $\lambda $  vanish on
$S\cap\{ |\overline{y}|\geq 1\}$. The denominator $2+2C$  is necessary
to ensure that $\mu $ and $\lambda $ are positive on
$S_+\cap\{\overline{y}=0\}$.

Consider now the formula $\Delta^{S}\mu  =\Delta^\top\mu  -H\,\mu _\nu$. Obviously
$|\mu _\nu|\leq\frac{1}{2+2C}+2<3$ and
$\Delta^\top\mu  \geq -2\, (n-1)=2-2\, n$, and so:
\[
\Delta^{S}\mu  \geq 2-2\,n-3\,
H\quad\mbox{\rm \ at any point of }\; S_+
\; .\]
This only prevents $\Delta^{S}\mu $
from being a large negative number. We need to choose $C_1$ so that
$C_1|\nabla^{S}\mu  |^2$  is large where $\lambda  >0$.  To estimate
$|\nabla^{S}\mu |$,  we use the unit length vector $\bf v$ which defines the steepest
direction in $S$.  At points where $\mu $ is positive:
\[
{\bf v}=\frac{1}{w}\,\left(\, \frac{\nabla
f}{|\nabla f|}\, ,\, |\nabla f|\,\right)
\;\Longrightarrow\;
 |\nabla^S\mu  |\;\geq\; |{\bf v}\,\mu  |
\; \geq\;
\frac{1}{w}\,\left(\,\frac{|\nabla f|}{2C}-2\,\right)
\; .\]
The last expression goes to  $1/(2C)$  as  $|\nabla f|\to\infty$. Thus
$C_1$ must be a multiple of  $C^2$   to
make $C_1\,|\nabla^{S}\mu |^2$  large. An easy calculation
shows:
\[
|\nabla f|\; >\; 20C\;\Longrightarrow\; \frac{1}{w}\,\left(\,
\frac{|\nabla f|}{2C}-2\,\right)\; >\; \frac{1}{3C}
\;\Longrightarrow\; |\nabla^S\mu| \; >\; \frac{1}{3C}
\; .\]
Accordingly we are going to choose $C_1$ satisfying
$C_1\cdot\left(\frac{1}{3C}\right)^2\geq 2\, n+3\, H$,
i.e.\ $C_1\geq (18\, n+27\, H)\, C^2\,$. With this choice, at
any point where $\mu $ and $\lambda $ are positive
(which certainly include the special points~$z(\varepsilon )$), we have:
\begin{eqnarray*}
&&|\nabla f|  >  20C \;\Longrightarrow\;
\Delta^{S}\mu +C_1\cdot |\nabla^{S}\mu  |^2
> 2-2\,n -3\, H+2\, n+3\, H= 2\;\Longrightarrow \\
&& \;\;\Longrightarrow\;
\Delta^{S}\lambda \;>\; 2\, C_1\, e^{C_1\mu }\; >\;
C_1+C_1\,\lambda \;\Longrightarrow\;\\
&& \;\;\Longrightarrow\;
\Delta^{S}\lambda  -R_0\,\lambda  \;>\; C_1
+(\, C_1-R_0\, )\,\lambda
\; .\end{eqnarray*}
Fix $\mu $ and $\lambda $  by choosing $C_1=\max\,\big(\,
(18\, n+27\, H)\, C^2\, ,\, R_0\,\big)$. For such a choice,  and for
$\varepsilon$  such that
$\displaystyle \frac{1}{\varepsilon}\,{\mathcal E}_1<C_1$,
the special points  $z(\varepsilon )$  satisfy $|\nabla f\big( z(\varepsilon )\big) |\leq 20\,C$ and
$w\big( z(\varepsilon )\big) \leq 1+20\, C$. Making now $\varepsilon\to 0$, we conclude that for $z\in S_+$  with $\lambda
(z)\neq 0$  it is:
\[
w(z)\;\leq\; \frac{1}{\lambda  (z)}\cdot\left(\max_{S_+}\,\lambda \right)\cdot (1+20\, C)
\; \leq\; \frac{e^{C_1}}{\lambda  (z)}\cdot (1+20\, C)
\; .\]
If $z=(x,0,y_n)\in S_+$  is any point with
$\overline{y}=0$, then $\mu  (z)\geq 1/2$ and
$\lambda  (z)\geq e^{C_1/2}-1$.  The desired estimate is then:
\begin{equation}\label{slope-estimate}
|\nabla u (x)|\; =\; |\nabla f(x,0)|\; <\; \frac{e^{C_1}}{e^{C_1/2}-1}(1+20C)\, ,\;\;
\mbox{\rm for all  }\; x\in X
\; .\end{equation}

The proof for $n=1$ is almost the same, with some
tiny simplifications that we next explain. Now we do not need to define $f$,
because $S$ is already the disjoint union of two Cartesian graphs $\{ y_1=\pm u(x)\}$.
We give $\lambda$ by the same ansatz as before, and $\mu$ by this one:
\[ \mu\; =\; 1-\frac{(y_1-\min u )^+}{2+2\, C}  \; .\]
This time we have $|\mu_\nu |<1$ and $\Delta^\top\mu =0$, hence $\Delta^S\mu=0-H\mu_\nu\geq -H$. We
choose $C_1$ satisfying $C_1\cdot\left(\frac{1}{3\, C}\right)^2\geq 2+H$ and $C_1\geq R_0$. Under these conditions,
at points where $|\nabla u|<20\, C$ we have:
\[ 
\Delta^S\mu+C_1\, |\nabla^S\mu |^2\;\geq\; -H+2+H\; =\; 2
\quad\mbox{\rm and}\quad\Delta^S\lambda-R_0\,\lambda\; >\; C_1
\; ,\]
and we recover the estimate (\ref{slope-estimate}) with this new choice for $C_1$.

\end{document}